\documentclass[11pt]{article}

\usepackage{amsmath}
\usepackage{amssymb}
\usepackage{amsthm}
\usepackage{setspace}

\allowdisplaybreaks
\usepackage{parskip}
\usepackage{xcolor}  %% https://www.overleaf.com/learn/latex/Using_colours_in_LaTeX
\usepackage{graphicx}
\usepackage{subcaption}
\usepackage{multicol}
\usepackage{multirow}
\usepackage[sort,round]{natbib}
\usepackage{lmodern}

\newtheorem{theorem}{Theorem}

\newtheorem{corollary}[theorem]{Corollary}

\newtheorem{proposition}[theorem]{Proposition}
\newtheorem{remark}[theorem]{Remark}

\usepackage{anysize}
\marginsize{15mm}{15mm}{15mm}{30mm} % left right top bottom

\DeclareMathOperator{\TCE}{TCE}
\DeclareMathOperator{\TV}{TV}
\DeclareMathOperator{\TS}{TS}

\providecommand{\keywords}[1]{\textbf{\textit{Keywords: }} #1}

\begin{document}

\title{Risk measures for infinite-server queues driven by a Cox-BESQ process}

\author{Jos\'{e} Da Fonseca\thanks{Auckland University of Technology, Business School, Department of Economics and Finance, Private Bag 92006, 1142 Auckland, New Zealand. Phone: +64 9 9219999 extn 5063. Email: jose.dafonseca@aut.ac.nz and PRISM Sorbonne EA 4101, Universit\'e Paris 1 Panth\'eon - Sorbonne, 17 rue de la Sorbonne, 75005 Paris, France.  ORCID: 0000-0002-6882-4511} 
\and
Yajun Shi\thanks{Auckland University of Technology,  Department of Mathematical Sciences, Private Bag 92006, 1142 Auckland, New Zealand. Email: yajun.shi@autuni.ac.nz. ORICD: 0009-0007-5262-2724}
\and
Wenjun Zhang\thanks{Corresponding author: Auckland University of Technology,  Department of Mathematical Sciences, Private Bag 92006, 1142 Auckland, New Zealand. Phone: +64 9 9219999 extn 5094. Email: wenjun.zhang@aut.ac.nz. ORICD: 0000-0001-7992-8204}
}

\date{\today}
\maketitle 

\singlespacing
%\onehalfspacing 
%\doublespacing

\begin{abstract}
This paper analyses infinite-server queues driven by a Cox process whose stochastic arrival intensity follows a squared Bessel (BESQ) process. Exploiting the affine structure of the model and previously derived transform representations for infinite-server queue variables, we derive the joint moment-generating function (MGF) of the arrival intensity, the number of customers in the system, and the time integral of the queue length at two or more dates. This enables a detailed analysis of the temporal evolution of infinite-server queues driven by a Cox–BESQ process. The availability of these explicit MGFs enables the development of risk measures for these variables, such as the joint moments of the arrival intensity, the number of customers in the system, and the integral of the number of customers in the system over a time interval, conditional on one of these variables exceeding a critical threshold. These risk measures can be determined at one or two dates. They rely essentially on the explicit form of the MGF of the state variables and on Fourier representations of the risk measures. Numerical experiments demonstrate the practical implementation of the proposed methodology and highlight the relevance of the theoretical results.
\end{abstract}


\setcounter{page}{1}

\section{Introduction}

This paper studies an infinite-server queue driven by a Cox process whose stochastic intensity is specified as a squared Bessel process (BESQ). Owing to the analytical properties of the model, and in particular its affine structure and explicit solvability, \citet{DaFonsecaShiZhang2026b} derived the joint moment-generating function (MGF) of the model state variables: the arrival intensity, the number of customers in the system, and the time integral of the number of customers in the system over a given interval. Building on this affine structure, the present paper shows that the MGF of these state variables at two or more distinct dates can also be obtained explicitly.

These explicit MGFs make it possible to derive closed-form representations for risk measures associated with the state variables. In particular, we consider expected joint moments conditional on one component of the state vector---namely, the arrival intensity, the number of customers in the system, or the cumulative number of customers in the system over a time interval---exceeding a critical threshold. These risk measures can be extended to two or more dates, thereby providing a detailed description of the behaviour of the system under critical conditions over time.

The proposed model offers an interesting alternative to infinite-server queues driven by Hawkes processes, as studied in \citet{DawPender2018} and \citet{KoopsSaxenaBoxmaMandjes2018}.\footnote{Note that \citet{DawPender2018} consider infinite-server queues driven by a scalar Hawkes process in which customers pass through different phases before leaving the system. Whenever we refer to their work, we refer specifically to their use of a Hawkes process as the arrival intensity and to their use of the infinitesimal generator to derive the main results. Although this description may appear reductive, these aspects are of both practical and theoretical importance.} 
Even though the model does not possess the self-exciting behaviour of a Hawkes process, it shares some of its statistical properties. In contrast, however, its analytical tractability allows for applications that may be difficult to develop within Hawkes-based frameworks.

The remainder of the paper is organised as follows. Section~\ref{sec:model} introduces the model, specifies its dynamics, and derives the MGFs at two or more distinct dates together with their mathematical properties. Section~\ref{sec:risk_management_of_queue} develops risk measures for infinite-server queues based on the arrival intensity, the number of customers in the system, and the cumulative number of customers in the system over a time interval. Section~\ref{sec:numerical_experiments} presents a range of numerical examples illustrating the proposed methodology, while Section~\ref{sec:Conclusion} concludes the paper.

\section{The model}\label{sec:model}

\subsection{Dynamics and simple properties of Cox-BESQ driven infinite-server queues}

Let \((\Omega, \mathcal{F}, \mathbb{P})\) be a complete probability space. Assume a Cox process \((N_t)_{t \geq 0}\) with stochastic intensity \((\lambda_t)_{t\geq 0}\) solving the stochastic differential equation:
\begin{equation}
	d\lambda_{t}=\kappa(\theta-\lambda_{t})dt+\sigma\sqrt{\lambda_{t}}dw_{t}, \label{eq:intensity}
\end{equation}	
with \((w_t)_{t \geq 0}\) a standard scalar Brownian motion under the probability measure \(\mathbb{P}\), and \(\kappa, \theta, \sigma \in \mathbb{R}_+\). We denote by \(\mathbb{E}[\cdot]\) the expectation under the probability measure \(\mathbb{P}\) while \(\mathbb{E}_t[\cdot] = \mathbb{E}[\cdot | \mathcal{F}_t] \) stands for the conditional expectation under this probability measure. Throughtout the work we assume the Feller condition  \(2\kappa \theta \geq \sigma^2\) is satisfied. Assume infinite-server queues with arrival rate given by \((N_t)_{t \geq 0}\), denote by \((q_t)_{t \geq 0}\) the process of the number of customers in the system. This process jumps positively by one unit when \(N_t\) jumps by one unit, it is the arrival process, and decreases at time \(t\) by one unit following a jump process, independent of the arrival jump process, with intensity \(\mu q_t\) with \(\mu>0\). 

The process \eqref{eq:intensity} is commonly referred to in finance as the CIR process when used for interest-rate modelling and as a square-root process (due to the form of the diffusion term) in equity-derivatives modelling. However, since it is a transformed BESQ process, see \citet[Proposition~6.3.1.1]{JeanblancYorChesney2009}, we will refer to it as a BESQ process. For the Cox process, we refer to \citet[p. 21]{Bremaud1981}

Denote the process \((x_t)_{t \geq 0}=(\lambda_t,q_t)_{t \geq 0}\in \mathbb{R}_+\times \mathbb{N}\) then its infinitesimal generator is defined for a function \(f:\mathbb{R}_+\times \mathbb{N}\to \mathbb{R}\) by
\begin{align}
	\mathcal{L}f(x)=\kappa(\theta-\lambda)\partial_\lambda f + \frac{\sigma^2}{2}\lambda \partial_{\lambda\lambda}^2f + \lambda \Big(f(\lambda,q+1)-f(x)\Big) + \mu q \Big(f(\lambda,q-1)-f(x)\Big), \label{eq:infinitesimal_generator}
\end{align}
and Feynman-Kac theorem gives for a function \(f:\mathbb{R}_+\times \mathbb{N}\to \mathbb{R}\) that the function \(t \to \mathbb{E}[f(x_t)]\) solves the partial differential equation (PDE): 
\begin{align}
	\frac{d\mathbb{E}[f(x_t)]}{dt}&= \mathbb{E}[\mathcal{L}f(x_t)]. \label{eq:PDE}
\end{align}

We follow \citet{DawPender2018} and make extensive use of the infinitesimal generator to derive the joint moments of \(q_t\) and \(\lambda_t\), as they did for infinite-server queues driven by a Hawkes process. As in their case, the model proposed here is affine, one consequence of which is that polynomial functions are preserved by the infinitesimal generator \eqref{eq:infinitesimal_generator}. In contrast to queues driven by a Hawkes process, however, the properties of the proposed model also allow us to derive explicitly several MGFs, which provide the key tools for developing analytical risk measures for \((\lambda_t, q_t, \int_0^t q_s ds)\).

Note that when \(\mu=0\), we have \(q_t=N_t\), provided that \(q_0=N_0\). Moreover, it is straightforward to prove, see \citet[Remark 2.3]{DFM2021}, that the first two moments of the number of jumps over a fixed time interval for a Cox process with BESQ intensity and those of the number of jumps over a fixed time interval for a (scalar) Hawkes process with an exponential kernel are identical when the parameters of the two processes satisfy certain equality constraints. This implies that, as far as the first two moments are concerned, the two processes are indistinguishable.

The two processes are, of course, very different. In the Hawkes process, the intensity depends on the number of past jumps through the well-known self-exciting property, whereas in the Cox process with BESQ intensity, the intensity is an autonomous process. In the Hawkes process, the intensity is piecewise continuous and highly regular between two jumps, following an exponential function of time, whereas in the Cox process with BESQ intensity, the intensity satisfies the SDE \eqref{eq:intensity}. The Cox process with BESQ intensity is an example of a process similar, as far as only the first two moments are concerned, to the Hawkes process, as mentioned by \citet[p.~83]{Hawkes1971}.

The fact that the Cox process with BESQ intensity, and by extension infinite-server queues with BESQ arrival intensity, does not possess the attractive self-exciting property of the Hawkes process may not be as severe a limitation as it might appear at first sight when one considers its analytical tractability, particularly the availability of several MGFs and the applications they enable. As we shall see, these MGFs can be derived easily thanks to the affine property of the Cox process with BESQ intensity, and, more generally, of infinite-server queues with BESQ arrival intensity, together with the extensive use of the infinitesimal generator.

The following result is straightforward to prove by applying \eqref{eq:PDE} to the function \(f(x)=\lambda^p q^m\) for \(p,m\in \mathbb{N}\). It relies on the property that affine processes are stable under polynomial functions; see \citet{CuchieroKeller-ResselTeichmann2012}.\footnote{The Hawkes process used in \citet{DawPender2018} and \citet{KoopsSaxenaBoxmaMandjes2018} is also affine and enables the derivation of the evolution of its moments; see also \citet{DassiosZhao2011}, \citet{DaFonsecaZaatour2014}, and \citet{CuiHawkesYi2020}, among others.}
\begin{proposition}\label{prop:moment_lambda_q}
	Assume a process \((x_t)_{t\in \mathbb{R}_+}\) with an infinitesimal generator of the form \eqref{eq:infinitesimal_generator}, then the moment \(\mathbb{E}[\lambda_t^pq_t^m]\) with \(p,m\in \mathbb{N}\) solves the ordinary differential equation (ODE):
	\begin{align}
		\frac{d\mathbb{E}[\lambda_t^pq_t^m]}{dt}&=\kappa\theta p\mathbb{E}[\lambda_{t}^{p-1}q_{t}^{m}]-\kappa p \mathbb{E}[\lambda_{t}^{p}q_{t}^{m}]+\frac{\sigma^{2}p(p-1)}{2}\mathbb{E}[\lambda_{t}^{p-1}q_{t}^{m}] \nonumber \\
		&+\sum_{i=1}^m\binom{m}{i} \mathbb{E}[\lambda_{t}^{p+1}q_t^{m-i}]+\mu \sum_{i=2}^m \binom{m}{i}(-1)^i\mathbb{E}[\lambda_{t}^{p}q_{t}^{m-i+1}]- \mu m\mathbb{E}[\lambda_{t}^{p}q_{t}^{m}]. \label{eq:moment_lambda_q}
	\end{align}
\end{proposition}

Although the ODEs have the same structure as those in \citet{DawPender2018} and could be analysed within the matrix framework of \citet{DawPender2023}, the fact that the intensity is an autonomous process simplifies the computation of these moments and reduces it to solving a set of scalar ODEs. Indeed, from \eqref{eq:moment_lambda_q}, to compute the ODE satisfied by \(\mathbb{E}[\lambda_t^p q_t^m]\), one needs to solve all the ODEs satisfied by the terms of the form \(\lbrace \mathbb{E}[\lambda_t^{p+i} q_t^{m-i}] ;\, i = 1,\ldots,m \rbrace\). Proceeding recursively leads to the computation of the ODEs satisfied by the terms \(\lbrace \mathbb{E}[\lambda_t^{k}] ;\, k = 1,\ldots,m+n \rbrace\). However, since \((\lambda_t)_{t \geq 0}\) is autonomous and affine, see \citet{DuffieKan1996}, it follows from \eqref{eq:moment_lambda_q} that, for \(k \in \mathbb{N}\), the expectation \(\mathbb{E}[\lambda_t^k]\) satisfies the ODE:
\begin{align}
	\frac{d\mathbb{E}[\lambda_t^k]}{dt}&=\kappa\theta k\mathbb{E}[\lambda_{t}^{k-1}]-\kappa k \mathbb{E}[\lambda_{t}^{k}]+\frac{\sigma^{2}k(k-1)}{2}\mathbb{E}[\lambda_{t}^{k-1}], \label{eq:moment_lambda}
\end{align}
which can be computed by solving first the scalar ODEs for \(\lbrace \mathbb{E}[\lambda_t^j], j =1,\ldots, k-1\rbrace\). 

If we further assume that the process \((\lambda_t)_{t \geq 0}\) is stationary, \textit{i.e.} \(\frac{d\mathbb{E}[\lambda_t^k]}{dt}=0\), from \eqref{eq:moment_lambda} we deduce that 
\begin{align*}
	\mathbb{E}[\lambda_{\infty}^{p}]=\frac{\prod_{i=2}^{p}(\kappa\theta+\frac{\sigma^{2}}{2}(i-1))}{\kappa^{p-1}}\mathbb{E}[\lambda_{\infty}],
\end{align*}
while \(\mathbb{E}[\lambda_{\infty}]=\theta\).	These are the moments of a noncentral chi-squared distribution. Still using \eqref{eq:moment_lambda_q}, and assuming the moments \(\mathbb{E}[q_t^m]\) exist, we deduce 
\begin{align}
	\label{Eqn:q_long_term}
	\mathbb{E}[q_{\infty}]&=\frac{1}{\mu}\mathbb{E}[\lambda_{\infty}]=\frac{\theta}{\mu},\\
	\mathbb{E}[q_{\infty}^{2}]&=\frac{\mathbb{E}[\lambda_{\infty}q_{\infty}]}{\mu} + \frac{\mathbb{E}[\lambda_{\infty}]+\mu\mathbb{E}[q_{\infty}]}{2\mu} = \frac{\mathbb{E}[\lambda_{\infty}q_{\infty}]}{\mu} +  \frac{\theta}{\mu}, \nonumber
\end{align}		
and since 
\begin{align}\label{Eqn:lamdba_q_long_term}
	\mathbb{E}[\lambda_{\infty}q_{\infty}]=\frac{2\kappa^{2}\theta^{2}+2\kappa\theta^{2}\mu+\sigma^{2}\theta\mu}{2\kappa\mu(\kappa+\mu)},
\end{align}
we obtain 
\begin{align*}
	\mathbb{E}[q_{\infty}^{2}]&=\frac{\theta^2}{\mu^2} +  \frac{\theta}{\mu} + \frac{\theta}{\mu}\frac{\sigma^2}{2\kappa(\kappa+\mu)}.
\end{align*}		
The first two moments of \(q_\infty\) suggests that it can be approximated by a Poisson distribution with rate \(\theta/\mu\), a fact that will be made precise later.

As in \citet{DawPender2018}, since the model is affine, we could also compute the moments \((q_{t+s} - q_t)\) for \(s>0\), or the auto-covariance function \(s \to \mathbb{E}[q_{t+s}q_{t}]\). 

To ensure that the MGFs derived in this work are well defined, we will use the following well known result for the scalar Riccati equation.

%%%%%%%%%%%%%%%%%%%%%%%%%%%%%%%%%%%%%%%%%%%%%%%%%%%%%%%%%%%%%%%%%%%%%%%%%%%%%%%%%%%%%%%%%%%%%%%%%%%%%%%%%%%%%%%%%%%%%%%%%%%%%%%%%%%%%%%%%%%%%%%%%
%%%%%%%%%%%%%%%%%%%%%%%%%%%%%%%%%%%%%%%%%%%%%%%%%%%%%%%%%%%%%%%%%%%%%%%%%%%%%%%%%%%%%%%%%%%%%%%%%%%%%%%%%%%%%%%%%%%%%%%%%%%%%%%%%%%%%%%%%%%%%%%%%
\subsection{Moment-generating functions}
A key property of infinite-server queues with an arrival rate driven by a Cox--BESQ process is that the MGF of
\(
\left(\lambda_t, q_t, \int_0^t q_u \, du\right)
\) is available in explicit form, as derived in \citet{DaFonsecaShiZhang2026b}. This explicit transform enables a range of applications, which we develop later in this paper. The result relies on the affine structure of the process and on the fact that the Riccati equation appearing in the MGF is explicitly solvable. Such solvability is not shared by all affine processes, notably Hawkes processes, but is a distinctive feature of the proposed model.

\begin{theorem}[Joint MGF of \((\lambda_T,q_T,\int_t^T q_sds)\), see \citet*{DaFonsecaShiZhang2026b}] \label{theo:mgf_triplet}
	Consider the joint MGF of  $\left(\lambda_T,q_{T},\int_t^{T} q_{s}ds\right)$ given for \(0\leq t < T\) and \(\mathbf{u}=(u_1,u_2,u_3)\in \mathbb{R}^3\) by
	\begin{align}
		f(t,T,\lambda_{t},q_{t})=\mathbb{E}_{t}\left[e^{u_{1}\lambda_{T} + u_{2}q_{T}+u_{3}\int_{t}^{T}q_{s}ds}\right], \label{eq:mgf_triplet_expectation}
	\end{align}
	then 
	\begin{align}
		f(t,T,\lambda_t,q_t)&=e^{a(T-t)+b(T-t)\lambda_{t}+d(T-t)q_{t}}, \label{eq:mgf_triplet} 
	\end{align}
	with
	\begin{align} 
		d(t)=\ln\left(e^{(u_{3}-\mu)t}e^{u_{2}}  -   \frac{\mu}{u_{3}-\mu}\left(1-e^{(u_{3}-\mu)t}\right)\right),\label{eq:sol_d_triplet}
	\end{align}
	and
	\begin{align} 
		b(t)=-\frac{2}{\sigma^2}\frac{\left(\frac{2}{\sigma^2} h_1'(0) + u_1h_1(0)\right)h_2'(t)- \left(\frac{2}{\sigma^2} h_2'(0) + u_1h_2(0)\right)h_1'(t)}{\left(\frac{2}{\sigma^2} h_1'(0) + u_1h_1(0)\right)h_2(t)- \left(\frac{2}{\sigma^2} h_2'(0) + u_1h_2(0)\right)h_1(t)}, \label{eq:sol_b_triplet}
	\end{align}
	where
	\begin{align}
		h_{1}(t)&= e^{\frac{-\kappa-\sqrt{\kappa^2+4\alpha}}{2} t}{}_0F_1\left(1+k_1, -\frac{\beta e^{(u_{3}-\mu)t}}{(\mu-u_{3})^{2}}\right), \label{eq:sol_h_1_triplet} \\
		h_{2}(t)&= e^{\frac{-\kappa+\sqrt{\kappa^2+4\alpha}}{2} t}{}_0F_1\left(1 - k_1, -\frac{\beta e^{(u_{3}-\mu)t}}{(\mu-u_{3})^{2}} \right), \label{eq:sol_h_2_triplet}
	\end{align}
	and
	\begin{align}
		h_{1}'(t)&=\frac{-\kappa-\sqrt{\kappa^{2}+4\alpha}}{2}h_{1}(t)+\frac{e^{\frac{-\kappa-\sqrt{\kappa^{2}+4\alpha}}{2}t}}{1+k_{1}}\frac{\beta e^{(u_{3}-\mu)t}}{|\mu-u_3|}{}_{0}F_{1}\left(2+k_{1},-\frac{\beta e^{(u_{3}-\mu)t}}{(\mu-u_{3})^{2}}\right), \label{eq:sol_h_1_triplet_prime} \\
		h_{2}'(t)&=\frac{-\kappa+\sqrt{\kappa^{2}+4\alpha}}{2}h_{2}(t)+\frac{e^{\frac{-\kappa+\sqrt{\kappa^{2}+4\alpha}}{2}t}}{1-k_{1}}\frac{\beta e^{(u_{3}-\mu)t}}{|\mu-u_{3}|}{}_{0}F_{1}\left(2-k_{1},-\frac{\beta e^{(u_{3}-\mu)t}}{(\mu-u_{3})^{2}}\right), \label{eq:sol_h_2_triplet_prime}
	\end{align}
	where
	\begin{align}
		k_{1}=\frac{\sqrt{\kappa^{2}+4\alpha}}{\mu-u_{3}}, \quad\alpha=\frac{\sigma^{2}}{2}\frac{u_{3}}{u_{3}-\mu}, \quad \beta=\frac{\sigma^{2}}{2}\left(e^{u_{2}}+\frac{\mu}{u_{3}-\mu}\right), \label{eq:k1_alpha_beta_triplet}
	\end{align}
	and \({}_0F_1(\nu,z)\) is the generalized hypergeometric function, see \citet[(16.2.1)]{DLMF}, while 
	\begin{equation}
		a(t)=-\frac{2\kappa \theta}{\sigma^2}\ln\left( \frac{\left(\frac{2}{\sigma^2} h_1'(0) + u_1h_1(0)\right)h_2(t)-\left(\frac{2}{\sigma^2} h_2'(0) + u_1h_2(0)\right)h_1(t)}{\left(\frac{2}{\sigma^2} h_1'(0) + u_1h_1(0)\right)h_2(0)-\left(\frac{2}{\sigma^2} h_2'(0) + u_1h_2(0)\right)h_1(0)}\right). \label{eq:sol_a_triplet}
	\end{equation}

	We may denote \(a(t,\mathbf{u})\), \(b(t,\mathbf{u})\), and \(d(t,\mathbf{u})\) when needed. Note that for \(\sqrt{\kappa^2+4\alpha}\) to be real, \(u_3\in ]-\infty,\; \mu/(1+2\sigma^2/\kappa^2)]\).
		
There exists a time \(t^*>0\) and positive thresholds \(u_1^*>0,\,u_2^*>0,\,u_3^*>0\), such that whenever \(u_1<u_1^*,\,u_2<u_2^*,\,u_3<u_3^*\) and \(t<t^*\) the MGF \eqref{eq:mgf_triplet} is well defined.
\end{theorem}
%%%%%%%%%%%%%%%%%%%%%%%%%%%%%%%%%%%%%%%%%%%%%%%%%%%%%%%%%%%%%%%%%%%%%%%%%%%%%%%%%%%%%%%%%%%%%%%%%%%%%%%%%%%%%%%%%%%%%%%%%%%%%%
%
% MGF of q_T
%
%%%%%%%%%%%%%%%%%%%%%%%%%%%%%%%%%%%%%%%%%%%%%%%%%%%%%%%%%%%%%%%%%%%%%%%%%%%%%%%%%%%%%%%%%%%%%%%%%%%%%%%%%%%%%%%%%%%%%%%%%%%%%%

The MGF of \(\lambda_t\) can be derived from Theorem \ref{theo:mgf_triplet}. It is well known and is reported here for convenience. This result can be found in the existing literature, such as Corollary 6.3.4.4 of \citet{JeanblancYorChesney2009}.
\begin{corollary}\label{coro:mgf_lambda}
	Denote for \(0\leq t < T\) and \(u \in \mathbb{R}\)
	\begin{align}
		f(t,T,\lambda_t)=\mathbb{E}_t\left[e^{u\lambda_T}\right], \label{eq:mgf_lambda_expectation}
	\end{align}
	the MGF of \(\lambda_T\) (conditional to time \(t\)) for a given \(T>0\). Then 
	\begin{align}
		f(t,T,\lambda_t)=e^{a(T-t)+ b(T-t)\lambda_t }, \label{eq:mgf_lambda}
	\end{align}
	with 
	\begin{align}
		d(t)&=0,\\
		b(t)&=\frac{u e^{-\kappa t}}{1-\frac{u\sigma^{2}}{2\kappa}(1-e^{-\kappa t})},\\
		a(t)&= -\frac{2\kappa\theta}{\sigma^{2}}\ln\left(1-\frac{u\sigma^{2}}{2\kappa}(1-e^{-\kappa t})\right).
	\end{align}
	For a given $u>0$, the explosion time \(t^*\) of \((\lambda_t)_{t \geq 0}\) is \(\frac{1}{-\kappa}\ln\left(1-\frac{2\kappa}{u\sigma ^2}\right).\)

	Note that \((a(t),b(t),d(t))\)  are such that \(a(t)=a(t,\mathbf{u})\), \(b(t)=b(t,\mathbf{u})\), \(d(t)=d(t,\mathbf{u})\), and \(\mathbf{u}=(u,0,0)\) with the right-hand side functions being those of Theorem \ref{theo:mgf_triplet}.	
\end{corollary}	

The MGF of \(q_t\) for a given time \(t>0\) can be derived from Theorem \ref{theo:mgf_triplet}.
\begin{corollary}\label{coro:mgf_q}
	Assume that \(\kappa/\mu \notin \mathbb{N}\) and for \(0\leq t < T\) define the MGF 
	\begin{align}
		f(t,T,\lambda_t,q_t)=\mathbb{E}_t\left[e^{uq_T}\right], \label{eq:mgf_q_expectation}
	\end{align}
	then 
	\begin{align}
		f(t,T,\lambda_t,q_t)=e^{a(T-t)+ b(T-t)\lambda_t + d(T-t)q_t},\label{eq:mgf_q}
	\end{align}
	with 
	\begin{align}
		d(t)&=\ln(1+e^{-\mu t}(e^{u}-1)),   \label{eq:sol_d_q} \\
		b(t)&=-\frac{2}{\sigma^2}\frac{h_1'(0)h_2'(t)-h_2'(0)h_1'(t)}{h_1'(0)h_2(t)-h_2'(0)h_1(t)}, \label{eq:sol_b_q} \\
		a(t)&=-\frac{2\kappa \theta}{\sigma^2}\ln\left( \frac{h_1'(0)h_2(t)-h_2'(0)h_1(t)}{h_1'(0)h_2(0)-h_2'(0)h_1(0)}\right), \label{eq:sol_a_q}
	\end{align}
	where
	\begin{align}
		h_{1}(t)&= e^{-\kappa t}{}_0F_1\left(1+\kappa/\mu, -\frac{\beta}{\mu^2}e^{-\mu t}\right), \label{eq:sol_h_1_q}\\
		h_{2}(t)&= {}_0F_1\left(1-\kappa/\mu, -\frac{\beta}{\mu^2}e^{-\mu t}\right),\\ %\label{eq:sol_h_2_q}
		h_{1}'(t)&=-\kappa h_{1}(t) + \frac{e^{-(\kappa+\mu) t}}{(1+\kappa/\mu)}\frac{\beta}{\mu}{}_0F_1\left(2+ \kappa/\mu, -\frac{\beta}{\mu^2}e^{-\mu t}\right),\\  %\label{eq:sol_h_1_q_prime} \\
		h_{2}'(t)&=  \frac{e^{-\mu t}}{(1-\kappa/\mu)}\frac{\beta}{\mu} {}_0F_1\left(2-\kappa/\mu, -\frac{\beta}{\mu^2}e^{-\mu t}\right), %\label{eq:sol_h_2_q_prime}
	\end{align}
	and
	\begin{align}
		\beta=\frac{\sigma^2}{2}(e^{u}-1). \label{eq:beta_q}
	\end{align}
	
There exist \(t^*>0\) and \(u^*>0\) such that, whenever \(u<u^*\) and \(t<t^*\), the MGF \eqref{eq:mgf_q_expectation} is well defined.

Note that \((a(t),b(t),d(t))\) of (\ref{eq:sol_d_q}-\ref{eq:sol_a_q}) are such that \(a(t)=a(t,\mathbf{u})\), \(b(t)=b(t,\mathbf{u})\), \(d(t)=d(t,\mathbf{u})\), and \(\mathbf{u}=(0,u,0)\) with the right-hand side functions being those of Theorem \ref{theo:mgf_triplet}.
\end{corollary}

\begin{corollary}\label{coro:q_infty}
	Assume that \(\kappa/\mu \notin \mathbb{N}\) and for \(0\leq t < T\) define the MGF
	\begin{align}
		f_\infty(u)=\lim_{T\to \infty}\mathbb{E}_t\left[e^{uq_T}\right], \label{eq:mgf_q_infty}
	\end{align}
	then 
	\begin{align}
		f_\infty(u)=e^{a(\infty)},\label{eq:mgf_q_infty_sol}
	\end{align}
	with 
	\begin{align}
		a(\infty)=\frac{2\kappa \theta}{\sigma^2}\ln\left( h_2(0) -  \frac{h_2'(0)h_1(0)}{h_1'(0)}\right), \label{eq:sol_a_infty}
	\end{align}
	where \(h_1(0), h_1'(0), h_2(0)\) and \(h_2'(0)\) are given by Corollary \ref{coro:mgf_q}. A series expansion of  \eqref{eq:sol_a_infty} with respect to \(\beta\) gives \(a(\infty)=\frac{\theta}{\mu}(e^{u}-1) + O(\beta^2)\). 
\end{corollary}

\begin{remark}
	The last part of Corollary \ref{coro:q_infty} implies that \(q_\infty\) has a distribution close to a Poisson distribution with rate \(\frac{\theta}{\mu}\).
\end{remark}

\begin{corollary}\label{coro:mgf_integral_q}
	Consider the MGF of \(\int_t^Tq_sds\) with \(0\leq t < T\) so that in Theorem \ref{theo:mgf_triplet} the following equality holds \(\beta=\alpha= \frac{\sigma^2}{2}\frac{u}{u-\mu}\) and assume \(\mathbb{R} \ni u<\mu/(1+2\sigma^2/\kappa^2)\). Then
	\begin{align}
		f(t,T,\lambda_t,q_t)&=\mathbb{E}_t\left[e^{u\int_t^Tq_s ds}\right]=e^{a(T-t)+ b(T-t)\lambda_t + d(T-t)q_t},\label{eq:mgf_integral_q}
	\end{align}
	with 
	\begin{align}
		d(t)=\ln\left(1- \frac{u}{u-\mu}(1-e^{(u-\mu) t})\right), 
		\label{eq:sol_d_integral_q}
	\end{align}
	
	\begin{align}
		b(t)=-\frac{2}{\sigma^2}\frac{h_1'(0)h_2'(t)-h_2'(0)h_1'(t)}{h_1'(0)h_2(t)-h_2'(0)h_1(t)}, \label{eq:sol_b_integral_q}
	\end{align}
	where
	\begin{align}
		h_{1}(t)&= e^{\frac{-\kappa-\sqrt{\kappa^2+4\alpha}}{2} t}{}_0F_1\left(1+k_1, -\frac{\alpha e^{(u-\mu)t}}{(\mu-u)^{2}}\right), \label{eq:sol_h_1_integral_q}\\
		h_{2}(t)&= e^{\frac{-\kappa+\sqrt{\kappa^2+4\alpha}}{2} t}{}_0F_1\left(1 - k_1, -\frac{\alpha e^{(u-\mu)t}}{(\mu-u)^{2}}\right), \label{eq:sol_h_2_integral_q}
	\end{align}
	with \(k_1=\frac{\sqrt{\kappa^2 + 4\alpha}}{\mu-u}\). The derivatives of these functions are given by
	\begin{align}
		h_{1}'(t)&=\frac{-\kappa-\sqrt{\kappa^2+4\alpha}}{2}h_{1}(t)  + \frac{e^{\frac{-\kappa-\sqrt{\kappa^2+4\alpha}}{2} t}}{1+k_1} \frac{\alpha e^{(u-\mu)t}}{|\mu-u|}{}_0F_1\left(2+k_1, -\frac{\alpha e^{(u-\mu)t}}{(\mu-u)^{2}}\right), \label{eq:sol_h_1_integral_q_prime} \\
		h_{2}'(t)&= \frac{-\kappa+\sqrt{\kappa^2+4\alpha}}{2}h_{2}(t)  + \frac{e^{\frac{-\kappa+\sqrt{\kappa^2+4\alpha}}{2} t}}{1-k_1} \frac{\alpha e^{(u-\mu)t}}{|\mu-u|}{}_0F_1\left(2-k_1, -\frac{\alpha e^{(u-\mu)t}}{(\mu-u)^{2}}\right), \label{eq:sol_h_2_integral_q_prime}
	\end{align}
	while
	\begin{align}
		a(t)=-\frac{2\kappa \theta}{\sigma^2}\ln\left( \frac{h_1'(0)h_2(t)-h_2'(0)h_1(t)}{h_1'(0)h_2(0)-h_2'(0)h_1(0)}\right). \label{eq:sol_a_integral_q}
	\end{align}

	There exist  \(t^*>0\) and  \(u^*>0\) such that, whenever \(u<u^*\) and \(t<t^*\), the MGF \eqref{eq:mgf_integral_q} is well defined.
\end{corollary}

\begin{remark}
From a numerical point of view, the MGF \eqref{eq:mgf_integral_q} is more tedious to manipulate since its argument is involved in the index of the special function \({}_0F_1\).
\end{remark}

\subsection{Moment-generating functions at two dates}
In the previous section, the joint MGF of \((\lambda_t, q_t, \int_0^{t} q_s ds)\) was obtained for a given date \(t > 0\). Therefore, the distribution of this random vector is known through the inverse Fourier transform of its characteristic function, and the expectation of a function of this vector can also be determined, provided that the Fourier transform of the function is known. The affine property enables much more than this single-date analysis. Indeed, the joint MGF of this vector at two distinct dates, and possibly at more dates, as the proof will show, can also be determined explicitly, as demonstrated by the following proposition.

\begin{proposition} \label{prop:two_dates_mgf}	
	Assume that \(\kappa/\mu \notin \mathbb{N}\), \(0 \leq t  \leq t_1 \leq t_2\), \((u_{1},u_{2},u_{3}) \in \mathbb{R}^3\) and \((\nu_{1},\nu_{2},\nu_{3}) \in \mathbb{R}^3\). Define 
	\begin{align}
		f(t,t_{1},t_{2},\lambda_{t},q_{t})=\mathbb{E}_{t}\left[e^{u_{1}\lambda_{t_{1}}+u_{2}q_{t_{1}}+u_{3}\int_{t}^{t_{1}}q_{s}ds+\nu_{1}\lambda_{t_{2}}+\nu_{2}q_{t_{2}}+\nu_{3}\int_{t}^{t_{2}}q_{s}ds}\right], 
		\label{eq:mgf_t1_t2_expectation} 
	\end{align}
	then 
	\begin{align}
		f(t,t_1,t_2,\lambda_t,q_t)&=e^{a(t_{1}-t)+b(t_{1}-t)\lambda_{t_{}}+d(t_{1}-t)q_{t_{}}}, \label{eq:mgf_t1_t2} 
	\end{align}
	with the functions \(a(t_{1}-t)\), \(b(t_{1}-t)\), and \(d(t_{1}-t)\) given by 
	\begin{align}
		a(t_1-t) &= a(t_{2}-t_{1},\mathbf{u}_2) + a(t_{1}-t, \mathbf{u}_1), \\
		b(t_1-t) &= b(t_{1}-t, \mathbf{u}_1),\\
		d(t_1-t) &= d(t_{1}-t,\mathbf{u}_1),
	\end{align}
	where the functions on the right-hand side are given in Theorem \ref{theo:mgf_triplet}, \(\mathbf{u}_2=(\nu_{1},\nu_{2},\nu_{3})\), and \(\mathbf{u}_1=\big(u_{1}+b(t_{2}-t_{1},\mathbf{u}_2),u_{2}+d(t_{2}-t_{1},\mathbf{u}_2),u_{3}+\nu _3\big)\).
\begin{proof}
	Using the tower property of the expectation and the affine property of the process, we obtain
	\begin{align*}
		f(t,t_{1},t_{2},\lambda_{t},q_{t})&=\mathbb{E}_{t}\left[e^{u_{1}\lambda_{t_{1}}+u_{2}q_{t_{1}}+u_{3}\int_{t}^{t_{1}}q_{s}ds+\nu_{3}\int_{t}^{t_{1}}q_{s}ds}\mathbb{E}_{t_{1}}\left[e^{\nu_{1}\lambda_{2}+\nu_{2}q_{t_{2}}+\nu_{3}\int_{t_{1}}^{t_{2}}q_{s}ds}\right]\right]\\
		&=\mathbb{E}_{t}\left[e^{u_{1}\lambda_{t_{1}}+u_{2}q_{t_{1}}+(u_{3}+\nu_{3})\int_{t}^{t_{1}}q_{s}ds}e^{a(t_{2}-t_{1},\mathbf{u}_2)+b(t_{2}-t_{1},\mathbf{u}_2)\lambda_{t_{1}}+d(t_{2}-t_{1},\mathbf{u}_2)q_{t_{1}}}\right], 
	\end{align*}
	where \(a(t_{2}-t_{1}, \mathbf{u}_2)\), \(b(t_{2}-t_{1}, \mathbf{u}_2)\), and \(d(t_{2}-t_{1}, \mathbf{u}_2)\) are the functions of Theorem \ref{theo:mgf_triplet} with \(\mathbf{u}_2=(\nu_{1},\nu_{2},\nu_{3})\), so that 
	\begin{align*}
		f(t,t_{1},t_{2},\lambda_{t},q_{t})&=e^{a(t_{2}-t_{1},\mathbf{u}_{2})}\mathbb{E}_{t}\left[e^{(u_{1}+b(t_{2}-t_{1},\mathbf{u}_{2}))\lambda_{t_{1}}+(u_{2}+d(t_{2}-t_{1},\mathbf{u}_{2}))q_{t_{1}}+(u_{3}+\nu_{3})\int_{t}^{t_{1}}q_{s}ds}\right]\\ 
		&=e^{a(t_{2}-t_{1},\mathbf{u}_2)}e^{a(t_{1}-t, \mathbf{u}_1)+ b(t_{1}-t, \mathbf{u}_1)\lambda_t + d(t_{1}-t,\mathbf{u}_1)q_t}, 
	\end{align*}
	with \(\mathbf{u}_1=\big(u_{1}+b(t_{2}-t_{1},\mathbf{u}_2),u_{2}+d(t_{2}-t_{1},\mathbf{u}_2),u_{3}+\nu_{3}\big)\) and the functions of Theorem \ref{theo:mgf_triplet}, which is the announced result.
\end{proof}		
\end{proposition}
The following immediate corollary is more to introduce specific cases that are developed in the applications.
\begin{corollary} \label{coro:qt_lambdat_two_dates}	
	Assume that \(\kappa/\mu \notin \mathbb{N}, 0 \leq t  \leq t_1 \leq t_2\) and \((u_{1},u_{2}) \in \mathbb{R}^2\) or \((u_{1},u_{2},u_{3}) \in \mathbb{R}^3\). The MGF of \((q_{t_1},q_{t_2})\) is given by   
	\begin{align}
		f(t,t_1,t_2,\lambda_t,q_t)=\mathbb{E}_{t}\left[e^{u_{1}q_{t_{1}}+u_{2}q_{t_{2}}}\right] = e^{a(t_{1}-t)+b(t_{1}-t)\lambda_{t_{}}+d(t_{1}-t)q_{t_{}}}, \label{eq:mgf_q_t1_t2}  
	\end{align}
	with the functions \(a(t_{1}-t)\), \(b(t_{1}-t)\), and \(d(t_{1}-t)\) given by
	\begin{align}
		a(t_1-t) &= a(t_{2}-t_{1},\mathbf{u}_2) + a(t_{1}-t, \mathbf{u}_1), \\
		b(t_1-t) &= b(t_{1}-t, \mathbf{u}_1),\\
		d(t_1-t) &= d(t_{1}-t,\mathbf{u}_1),
	\end{align}
	where the above functions on the right-hand side are given in Theorem \ref{theo:mgf_triplet} with \(\mathbf{u}_2=(0,u_2,0)\), and \(\mathbf{u}_1=\big(b(t_{2}-t_{1},\mathbf{u}_2),u_{1}+d(t_{2}-t_{1},\mathbf{u}_2),0\big)\).

	The MGF of \( (\lambda_{t_{1}}, \lambda_{t_{2}}) \) is given by   
	\begin{align}
		f(t,t_1,t_2,\lambda_t)=\mathbb{E}_{t}\left[e^{u_{1}\lambda_{t_{1}}+u_{2}\lambda_{t_{2}}}\right] = e^{a(t_{1}-t)+b(t_{1}-t)\lambda_{t_{}}}, \label{eq:mgf_lambda_t1_t2} 
	\end{align}
	with the functions \(a(t_{1}-t)\), \(b(t_{1}-t)\), and \(d(t_{1}-t)\) given by
	\begin{align}
		a(t_1-t) &= a(t_{2}-t_{1},\mathbf{u}_2) + a(t_{1}-t, \mathbf{u}_1), \\
		b(t_1-t) &= b(t_{1}-t, \mathbf{u}_1),\\
		d(t_1-t) &= 0,
	\end{align}
	where the above functions on the right-hand side are given in Theorem  \ref{theo:mgf_triplet} with \(\mathbf{u}_2=(u_2,0,0)\), and \(\mathbf{u}_1=\big(u_1+b(t_{2}-t_{1},\mathbf{u}_2),0,0\big)\).

	The MGF of \( (\lambda_{t_{1}}, \lambda_{t_{2}}, q_{t_2}) \) is given by   
	\begin{equation}
		f(t,t_1,t_2,\lambda_t,q_t)=\mathbb{E}_{t}\left[e^{u_{1}\lambda_{t_{1}}+u_{2}\lambda_{t_{2}}+u_{3}q_{t_{2}}}\right] = e^{a(t_{1}-t)+b(t_{1}-t)\lambda_{t_{}}+d(t_{1}-t)q_{t_{}}}, \label{eq:mgf_lambda_t1_lambda_t2_q_t2} 
	\end{equation}
	with the functions \(a(t_{1}-t)\), \(b(t_{1}-t)\), and \(d(t_{1}-t)\) given by
	\begin{align}
		a(t_1-t) &= a(t_{2}-t_{1},\mathbf{u}_2) + a(t_{1}-t, \mathbf{u}_1), \\
		b(t_1-t) &= b(t_{1}-t, \mathbf{u}_1),\\
		d(t_1-t) &= d(t_{1}-t,\mathbf{u}_1),
	\end{align}
	where the above functions on the right-hand side are given in Theorem \ref{theo:mgf_triplet}, \(\mathbf{u}_2=(u_2,u_3,0)\), and \(\mathbf{u}_1=\big(u_1+b(t_{2}-t_{1},\mathbf{u}_2),d(t_{2}-t_{1},\mathbf{u}_2),0\big)\).

	The MGF of \( (q_{t_{1}}, \lambda_{t_{2}}, q_{t_2}) \) is given by   
	\begin{align}
		f(t,t_1,t_2,\lambda_t,q_t)=\mathbb{E}_{t}\left[e^{u_{1}q_{t_{1}}+u_{2}\lambda_{t_{2}}+u_{3}q_{t_{2}}}\right] = e^{a(t_{1}-t)+b(t_{1}-t)\lambda_{t_{}}+d(t_{1}-t)q_{t_{}}}, \label{eq:mgf_q_lambda_q_t1_t2}
	\end{align}
	with the functions \(a(t_{1}-t)\), \(b(t_{1}-t)\), and \(d(t_{1}-t)\) given by
	\begin{align}
		a(t_1-t) &= a(t_{2}-t_{1},\mathbf{u}_2) + a(t_{1}-t, \mathbf{u}_1), \\
		b(t_1-t) &= b(t_{1}-t, \mathbf{u}_1),\\
		d(t_1-t) &= d(t_{1}-t,\mathbf{u}_1),
	\end{align}
	where the above functions on the right-hand side are given in Theorem \ref{theo:mgf_triplet}, \(\mathbf{u}_2=(u_2,u_3,0)\), and \(\mathbf{u}_1=\big(b(t_{2}-t_{1},\mathbf{u}_2),u_1+d(t_{2}-t_{1},\mathbf{u}_2),0\big)\).
\end{corollary}
\begin{remark}
It should be clear that it is possible to consider more than two dates and still obtain explicit formulas for these functions.
\end{remark}
%%%%%%%%%%%%%%%%%%%%%%%%%%%%%%%%%%%%%%%%%%%%%%%%%%%%%%%%%%%%%%%%%%%%%%%%%%%%%%%%%%%%%%%%%%%%%%%%%%%%%%%%%%%%%%%%%%%%%%%%%%%%%%%%%%%%%%%%%%%%%%%%%

\section{Risk management of Cox--BESQ driven infinite-server queues}\label{sec:risk_management_of_queue}

In this section we introduce a family of tail-based risk measures and develop tractable methods for their evaluation in the context of infinite-server queues. The analysis relies on Fourier transform techniques that allow truncated and conditional moments to be computed directly from the MGFs derived in the previous section.

We first introduce a class of conditional tail risk measures in Section~\ref{sec:ConditionaTail}. These measures characterise the behaviour of the system when the underlying quantity of interest, which in practice will be one of the variables \(\lambda_t\), \(q_t\), or \(\int_0^t q_s \, ds\), exceeds a given threshold. They include quantities such as the tail conditional expectation, tail variance, and tail skewness, which describe the magnitude, dispersion, and asymmetry of the tail distribution. These risk measures are extended to the multivariate case, combining several variables and quantifying joint moments conditional on one of the variables exceeding a given threshold.

The evaluation of these quantities requires the computation of truncated moments. Section~\ref{sec:MultiFourierTransform} demonstrates how such moments can be expressed in terms of one-dimensional or multidimensional Fourier transform representations. These representations provide a unified framework for computing tail-based quantities directly from the MGF of the underlying random variables.

However, in the multivariate case, multidimensional Fourier representations require high-dimensional numerical integration and are therefore computationally expensive, making the framework difficult to operationalise.

To address this limitation, Section~\ref{sec:MultiFourierTransform} discusses an alternative representation based on one-dimensional Fourier transforms. This reduction is achieved by expressing truncated mixed moments in terms of higher-order derivatives of the MGF. The resulting formulas significantly reduce computational complexity and are straightforward to implement.

These transform representations are derived in \citet{DaFonsecaWong2026}, we present them here for completeness and to emphasize their broader applicability. In particular, the methodology applies to any distribution with a known MGF and therefore extends far beyond Cox--BESQ infinite-server queues considered in this paper. Consequently, the framework is relevant for a wide range of models arising in applied probability, queueing theory, and financial mathematics. For this latter topic, see \citet{DuffiePanSingleton2000}.

The application of these results to infinite-server queues is presented in Section~\ref{sec:QueueRiskMeasures}, where explicit expressions for infinite-server queues related tail risk measures are obtained.

\subsection{Conditional tail risk measures}\label{sec:ConditionaTail}

Consider a non-negative random variable $Y$ representing a quantity of interest. In risk management, it is often important to characterize the behaviour of $Y$ when it exceeds a given threshold $y_*$. Such tail-based quantities are commonly described using conditional moments.

The conditional $p$-th moment of $Y$ given that $Y>y_*$ is defined as
\begin{align}
	\mathbb{E}[Y^p \mid Y > y_*]
	=
	\frac{\mathbb{E}[Y^p \mathbf{1}_{\{Y>y_*\}}]}
	{\mathbb{E}[\mathbf{1}_{\{Y>y_*\}}]},
	\label{eq:ConditionalMomentYp}
\end{align}
where $\mathbf{1}_{\{\cdot\}}$ denotes the indicator function.

Hence the computation of conditional moments reduces to evaluating truncated moments of the form
\[
\mathbb{E}[Y^p \mathbf{1}_{\{Y>y_*\}}].
\]
While such quantities are typically obtained using the density of \(Y\), the transform-based approach presented in this paper expresses them directly in terms of the MGF of \(Y\). This will be particularly relevant in our case, since neither the density of \(q_t\) nor that of \(\int_0^t q_s \, ds\) is known.

\paragraph{Tail conditional expectation (TCE).}

For $p=1$, the conditional moment \eqref{eq:ConditionalMomentYp} corresponds to the tail conditional expectation:
\begin{align}
	\TCE_Y(y_*)
	=
	\mathbb{E}[Y \mid Y>y_*]
	=
	\frac{\mathbb{E}[Y\mathbf{1}_{\{Y>y_*\}}]}
	{\mathbb{P}(Y>y_*)}.
	\label{eq:TailExpectation}
\end{align}

The TCE measures the average magnitude of $Y$ when the threshold $y_*$ is exceeded and is widely used in financial risk management under the name expected shortfall.

\paragraph{Tail variance (TV).}

The variability of the tail distribution can be quantified by the conditional variance:
\begin{align}
	\TV_Y(y_*)
	=
	\mathbb{E}\!\left[(Y-\mathbb{E}[Y\mid Y>y_*])^2 \mid Y>y_*\right]
	=
	\mathbb{E}[Y^2\mid Y>y_*]
	-
	\left(\mathbb{E}[Y\mid Y>y_*]\right)^2 .
	\label{eq:TailVariance}
\end{align}

\paragraph{Tail skewness (TS).}

The asymmetry of the tail distribution can be described by the standardized third central conditional moment:
\begin{align}
	\TS_Y(y_*)
	=
	\frac{
		\mathbb{E}\!\left[(Y-\mathbb{E}[Y\mid Y>y_*])^3 \mid Y>y_*\right]
	}{
		\left(
		\mathbb{E}\!\left[(Y-\mathbb{E}[Y\mid Y>y_*])^2 \mid Y>y_*\right]
		\right)^{3/2}
	}.
	\label{eq:tail_skew}
\end{align}

Once conditional moments $\mathbb{E}[Y^p\mid Y>y_*]$ are available for $p=1,2,3$, the above tail risk measures follow immediately.

\subsection{Fourier transform representations for risk measures}\label{sec:MultiFourierTransform}

We now describe how truncated moments can be expressed directly in terms of the MGF of the underlying random variable. The following result, adapted from \citet{DaFonsecaWong2026}, provides a Fourier representation for truncated moments of arbitrary order. We postpone to the next section the proofs of the result.

\begin{proposition}\label{prop:Expected_shortfall_Fourier}
	
	Let $Y$ be a random variable taking values in \(\mathbb{R}_+\) with MGF $\Phi_Y(v)=\mathbb{E}[e^{vY}]$. Fix $y_*\in \mathbb{R}_+$ and $\alpha_1>0$. Assume that for $j=0,\ldots,p$ the function
	
	\[
	u_1 \in \mathbb{R}_+ \mapsto
	\frac{\Phi_Y(\alpha_1-\mathrm{i}u_1)}
	{(\alpha_1-\mathrm{i}u_1)^{j+1}}
	\]
	
	is integrable. Then
	\begin{align}
		\mathbb{E}[Y^p\mathbf{1}_{\{Y>y_*\}}]
		=
		\int_0^{+\infty}
		\sum_{j=0}^{p}
		\binom{p}{j}
		\frac{y_*^{p-j}\Gamma(j+1)}{\pi}
		\Re\!\left(
		\frac{e^{-(\alpha_1-\mathrm{i}u_1)y_*}
			\Phi_Y(\alpha_1-\mathrm{i}u_1)}
		{(\alpha_1-\mathrm{i}u_1)^{j+1}}
		\right)
		du_1 .
		\label{eq:moment_Yp}
	\end{align}

	Let $Y$ be a random variable with values in \(\mathbb{N}\) with MGF $\Phi_Y(v)=\mathbb{E}[e^{vY}]$. Assume $y_*\in \mathbb{N}$,  then
	\begin{align}
		\mathbb{E}[Y^p\mathbf{1}_{\{Y>y_*\}}] =\frac{1}{2\pi}\int_{-\pi}^{+\pi} \Re \left(  \Phi_{Y}(\mathrm{i}u_1)\left(\sum_{k=y_*+1}^{+\infty}k^pe^{- \mathrm{i} u_1 k}\right)\right)du_1\,. \label{eq:moment_Yp_N}
		% \mathbb{E}[Y^p\mathbf{1}_{\{Y>y_*\}}] =\frac{1}{2\pi}\int_{-\pi}^{+\pi} \Re \left( \Phi_{Y}(\mathrm{i}u_1)\left(
		% \frac{e^{-\mathrm{i}u_1}A_p(e^{-\mathrm{i}u_1})}
		% {\left(1-e^{-\mathrm{i}u_1}\right)^{p+1}}
		% -
		% \sum_{k=1}^{y_*} k^p e^{-\mathrm{i}u_1k}
		% \right)\right)du_1\,,
	\end{align}
\end{proposition}

In many applications, it is necessary to compute mixed moments involving several variables. For instance, given a pair of random variables $(Y,Z)$, one may need to evaluate \(\mathbb{E}\!\left[ Z^r Y^p \mid Y > y_* \right],\)
for $r,p \in \mathbb{N}$. This quantity can be expressed in terms of the truncated mixed moment
\begin{align}
\mathbb{E}\left[ Z^r Y^p | Y > y_* \right] &= \frac{\mathbb{E}\left[ Z^r Y^p \mathbf{1}_{\{Y>y*\}} \right]}{\mathbb{E}\left[ \mathbf{1}_{\{Y>y*\}} \right]}, \label{eq:conditional_moment_ZrYp}
\end{align}
which requires knowledge of the joint density of the vector $(Y,Z)$. For the model proposed in this work, this density is unknown, but its MGF is available. As a result, we need to express this expectation in terms of the MGF, which is the purpose of the following proposition.

\begin{proposition}\label{prop:JointExpectedShortfallFourierTransform}
	
	Let $(Y,Z)$ be a two-dimensional random vector with values in \(\mathbb{R}_+\times \mathbb{R}_+\) with joint MGF denoted by
	\begin{align}
	\Phi_{(Y,Z)}(v_1,v_2) = \mathbb{E}[e^{v_1Y+v_2Z}]. \label{eq:joint_mgf}
	\end{align}
	
	Fix $y_*\in \mathbb{R}_+$, $\alpha_1, \alpha_2 \in \mathbb{R}_+$ and \(r,p\in \mathbb{N}\). Under suitable integrability conditions we have the integral representation
	\begin{align}
		\mathbb{E}[Z^r Y^p \mathbf{1}_{\{Y>y_*\}}]
		=
		\int_0^{+\infty}\!\!\int_0^{+\infty}
		\sum_{j=0}^{p}
		\binom{p}{j}
		\frac{y_*^{p-j}\Gamma(j+1)\Gamma(r+1)}{\pi^2}
		\Re\!\left(
		\frac{
			e^{-(\alpha_1-\mathrm{i}u_1)y_*}
			\Phi_{(Y,Z)}(\alpha_1-\mathrm{i}u_1,\alpha_2-\mathrm{i}u_2)
		}{
			(\alpha_1-\mathrm{i}u_1)^{j+1}
			(\alpha_2-\mathrm{i}u_2)^{r+1}
		}
		\right)
		du_1du_2 .
		\label{eq:expectation_ZqYp}
	\end{align}
	
If $(Y,Z)$ be a two-dimensional random vector with values in \(\mathbb{N}\times \mathbb{R}_+\) with a joint MGF denoted by \eqref{eq:joint_mgf} and let $y_*\in \mathbb{N}$ and \(r,p\in \mathbb{N}\). Under suitable integrability conditions we have the integral representation
	\begin{align}
		\mathbb{E}[Z^r Y^p \mathbf{1}_{\{Y>y_*\}}]
		=
		\int_0^{+\infty}\!\!\int_{-\pi}^{+\pi}
		\sum_{k=y_* + 1}^{+\infty}
		\frac{\Gamma(r+1)}{2\pi^2}
		\Re\!\left(
		\frac{
			k^pe^{-\mathrm{i}ku_1}
			\Phi_{(Y,Z)}(\mathrm{i}u_1,\alpha_2-\mathrm{i}u_2)
		}{
			(\alpha_2-\mathrm{i}u_2)^{r+1}
		}
		\right)
		du_2 .
		\label{eq:expectation_ZqYp_NxR}
	\end{align}
	
If $(Y,Z)$ be a two-dimensional random vector with values in \(\mathbb{N}\times \mathbb{N}\) with a joint MGF denoted by \eqref{eq:joint_mgf} and let $y_*\in \mathbb{N}$ and \(r,p\in \mathbb{N}\). Under suitable integrability conditions we have the integral representation
	\begin{align}
		\mathbb{E}[Z^r Y^p \mathbf{1}_{\{Y>y_*\}}]
		=
		\int_{-\pi}^{+\pi}\!\!\int_{-\pi}^{+\pi}
		\sum_{k_1=y_* + 1}^{+\infty}
		\sum_{k_2 = 0}^{+\infty}
		\frac{1}{4\pi^2}
		\Re\left(
		k_1^pk_2^r  e^{-\mathrm{i}k_1u_1}e^{-\mathrm{i}k_2u_2}
		\Phi_{(Y,Z)}(\mathrm{i}u_1,\mathrm{i}u_2)
		\right).
		\label{eq:expectation_ZqYp_NxN}
	\end{align}

\end{proposition}

%\subsection{One-dimensional Fourier transform representations for risk measures}\label{sec:OneFourierTransform}

The multidimensional Fourier representation above can be simplified significantly by exploiting derivatives of the MGF. This reduction leads to one-dimensional Fourier integrals and therefore greatly improves numerical efficiency.

\begin{proposition}\label{prop:JointExpectedShortfallDerivativeFourierTransform}
	
	Let $(Y,Z)$ be a two-dimensional random vector taking values in \(\mathbb{R}^2\) or \(\mathbb{R} \times \mathbb{N}\) , with joint MGF denoted $\Phi_{(Y,Z)}(u_1,u_2)$. Denote by 	$\Phi_{(Y,Z)}^{(0,r)}(u_1,u_2)$ the $r$-th derivative with respect to $u_2$. Then for \(r,p,\in \mathbb{N}\) we have 
	\begin{align}
		\mathbb{E}[Z^r Y^p \mathbf{1}_{\{Y>y_*\}}]
		=
		\int_0^{+\infty}
		\sum_{j=0}^{p}
		\binom{p}{j}
		\frac{y_*^{p-j}\Gamma(j+1)}{\pi}
		\Re\!\left(
		\frac{
			e^{-(\alpha_1-\mathrm{i}u_1)y_*}
			\Phi_{(Y,Z)}^{(0,r)}(\alpha_1-\mathrm{i}u_1,0)
		}{
			(\alpha_1-\mathrm{i}u_1)^{j+1}
		}
		\right)
		du_1 .
		\label{eq:ExpectationZqYp-1dim}
	\end{align}
	
	If instead $(Y,Z)$ is a two-dimensional random vector with values in \(\mathbb{N}\times \mathbb{R}\) or \(\mathbb{N}\times \mathbb{N}\). Then for \(r,p,\in \mathbb{N}\) we have 
	\begin{align}
		\mathbb{E}[Z^r Y^p \mathbf{1}_{\{Y>y_*\}}] = \frac{1}{2\pi}\int_{-\pi}^{+\pi} \Re \left( \Phi_{(Y,Z)}^{(0,r)}(\mathrm{i}u_1,0)\left(\sum_{k=y_*+1}^{+\infty}k^pe^{- \mathrm{i} u_1 k}\right)\right)du_1\,.
		\label{eq:ExpectationZqYp-1dim-N}
	\end{align}
	
\end{proposition}

%\begin{proof}
%Consider \(\mathbb{E}[e^{\mathrm{i} u_1 Y+ u_2 Z}]=\Phi_{(Y,Z)}(\mathrm{i} u_1,u_2)\). Then 
%\begin{align*}
%\mathbb{E}[Z^r \mathbf{1}_{\{Y = k\}}] = \frac{1}{2\pi}\int_{-\pi}^{+\pi} \Re\left(\Phi_{(Y,Z)}^{(0,r)}(\mathrm{i}u_1,0)e^{- \mathrm{i} u_1 k}\right)du_1\,,
%\end{align*}
%and since 
	%\begin{align*}
		%\mathbb{E}[Z^r Y^p \mathbf{1}_{\{Y>y_*\}}]=\sum_{k=y_*+1}^{+\infty} k^p \mathbb{E}[Z^r \mathbf{1}_{\{Y = k\}}]
		%%\frac{1}{2\pi}\int_{-\pi}^{+\pi} \Phi_{(Y,Z)}^{(0,q)}( \mathrm{i}v_1,0)e^{- \mathrm{i} v_1 k}dv_1\,,
	%\end{align*}
%we obtain \eqref{eq:ExpectationZqYp-1dim-N}.
%
%\end{proof}
%

\subsection{Evaluation of risk measures of Cox-BESQ driven infinite-server queues}\label{sec:QueueRiskMeasures}

We now apply the above framework to infinite-server queues driven by a Cox--BESQ process. The quantities of interest involve the number of customers in the system $q_t$, its time-integrated counterpart $\bar{q}_t = \int_0^t q_s\,ds$ and the arrival intensity $\lambda_t$.

We consider the risk measures based on the three aformentioned variables, these are given by 
\[
\mathbb{E}\!\left[q_t^p \mid q_t>q_*\right] = \frac{\mathbb{E}\!\left[q_t^p \mathbf{1}_{\{q_t>q_*\}}\right] }{\mathbb{E}\!\left[\mathbf{1}_{\{q_t>q_*\}}\right] },
\qquad 
\mathbb{E}\!\left[\lambda_t^p \mid \lambda_t>\lambda_*\right] = \frac{\mathbb{E}\!\left[\lambda_t^p \mathbf{1}_{\{\lambda_t>\lambda_*\}}\right] }{\mathbb{E}\!\left[\mathbf{1}_{\{\lambda_t>\lambda_*\}}\right] }\,,
\qquad
 \mathbb{E}\!\left[\bar{q}_t^p \mid \bar{q}_t>\bar{q}_*\right]= \frac{\mathbb{E}\!\left[\bar{q}_t^p \mathbf{1}_{\{\bar{q}_t>\bar{q}_*\}}\right] }{\mathbb{E}\!\left[\mathbf{1}_{\{\bar{q}_t>\bar{q}_*\}}\right] }.
\]
The first one represents the $p$-th moment of the number in the system at time $t$ restricted to the event that the number of customers exceeds a prescribed threshold $q_*$. The second one represents the $p$-th moment of arrival intensity at time $t$ conditional on this intensity being above a prescribed threshold $\lambda_*$. The last one extends the first one to the number of customers in the system over a time interval $[0,\, t]$.

The quantities discussed above are the risk measures associated with the number of customers in the system, its time-integrated counterpart and the arrival intensity. Their evaluation relies crucially on the availability of closed-form expressions for the MGFs derived in the previous section. In particular, these MGFs allow us to derive tractable expressions for these conditional expectations. To compute the quantities above, the numerators need to be determined and this is the purpose of the following proposition.

\begin{proposition}\label{prop:qt_ExpectedShortfallFourierTransform}
	Let $q_t$ denote the number of customers in the system at time $t$, and \(u\in \mathbb{R} \to \Phi_{q_t}(u)=\mathbb{E}[e^{u q_t}]\) its MGF given by Corollary~\ref{coro:mgf_q}. Then \eqref{eq:moment_Yp_N} rewrites
	\begin{align}
		\mathbb{E}[q_t^p \mathbf{1}_{\{q_t>q_*\}}]
%		&= \frac{1}{2\pi}\int_{-\pi}^{+\pi} \Re \left( \Phi_{q_t}(\mathrm{i}u_1)\left(\sum_{k=q_*+1}^{+\infty}k^pe^{- \mathrm{i} u_1 k}\right)\right)du_1\,,
		&= \frac{1}{2\pi}\int_{-\pi}^{+\pi} \Re \left( \Phi_{q_t}(\mathrm{i}u_1)\left(\sum_{k=q_*+1}^{\infty} k^p  e^{-\mathrm{i}u_1k}\right)\right)du_1\, 
		\label{eq:moment_q_t_series}\\
		&=
		\frac{1}{2\pi}\int_{-\pi}^{+\pi} \Re \left( \Phi_{q_t}(\mathrm{i}u_1)\left(	\frac{e^{-\mathrm{i}u_1}A_p(e^{-\mathrm{i}u_1})}{\left(1-e^{-\mathrm{i}u_1}\right)^{p+1}}
		- \sum_{k=1}^{q_*} k^p e^{-\mathrm{i}u_1k} \right)\right)du_1\,,
		\label{eq:moment_q_t}
	\end{align}

	%with \(q_*\in \mathbb{N}\), \(r\in(0,1)\) and 
	where $A_p(\cdot)$ is the Eulerian-polynomial of order $p$.
	
	Let $\lambda_t$ denote the arrival intensity at time $t$, and \(u\in \mathbb{R} \to \Phi_{\lambda_t}(u)=\mathbb{E}[e^{u \lambda_t}]\) its MGF given by Corollary~\ref{coro:mgf_lambda}. Then \eqref{eq:moment_Yp} rewrites
	\begin{align}
		\mathbb{E}[\lambda_t^p\mathbf{1}_{\{\lambda_t>\lambda_*\}}]
		=
		\int_0^{+\infty} \sum_{j=0}^{p} \binom{p}{j} \frac{\lambda_*^{p-j}\Gamma(j+1)}{\pi} \Re\!\left( \frac{e^{-(\alpha_1-\mathrm{i}u_1)\lambda_*} \Phi_{\lambda_t}(\alpha_1-\mathrm{i}u_1)}
		{(\alpha_1-\mathrm{i}u_1)^{j+1}}
		\right)
		du_1 \,,
		\label{eq:moment_lambda_t}
	\end{align}

	with \(\lambda_*\in \mathbb{R}_+\).
	
	Let $\bar{q}_t$ denote the number of customers in the system over the time interval \([0,\; t]\), and \(u\in \mathbb{R} \to \Phi_{\bar{q}_t}(u)=\mathbb{E}[e^{u \int_0^t q_sds}]\) its MGF given by Corollary~\ref{coro:mgf_integral_q}. Then \eqref{eq:moment_Yp} rewrites
	\begin{align}
		\mathbb{E}[\bar{q}_t^p\mathbf{1}_{\{\bar{q}_t>\bar{q}_*\}}]
		=
		\int_0^{+\infty} \sum_{j=0}^{p} \binom{p}{j} \frac{{\bar{q}_*}^{p-j}\Gamma(j+1)}{\pi} \Re\!\left( \frac{e^{-(\alpha_1-\mathrm{i}u_1)\bar{q}_*} \Phi_{\bar{q}_t}(\alpha_1-\mathrm{i}u_1)}
		{(\alpha_1-\mathrm{i}u_1)^{j+1}}
		\right)
		du_1 \,,
		\label{eq:moment_bar_q_t}
	\end{align}
	
	with \(\bar{q}_*\in \mathbb{R}_+\).
	
	%\begin{align}
	%\mathbb{E}[q_t^p \mathbf{1}_{\{q_t>q_*\}}]
	%&=
	%%q_*^p\,\mathbb{P}(q_t>q_*)  \nonumber\\
	%q_*^p	\left(\frac12+\frac{1}{\pi}\int_0^\infty
	%\Im\!\left(
	%\frac{e^{-iu_1q_*}\Phi_{q_t}(iu_1)}{u_1}
	%\right)\,du_1
	%\right)		\nonumber\\
	%&\quad+
	%\int_0^{+\infty}
	%\sum_{j=1}^{p}
	%\binom{p}{j}
	%\frac{q_*^{p-j}\Gamma(j+1)}{\pi}
	%\Re\!\left(
	%\frac{e^{-(\alpha_1-\mathrm{i}u_1)q_*}
		%\Phi_{q_t}(\alpha_1-\mathrm{i}u_1)}
	%{(\alpha_1-\mathrm{i}u_1)^{j+1}}
	%\right)du_1.
	%\label{eq:Momentq_t}
	%\end{align}
	%
\end{proposition}

\begin{proof}
	We first consider the discrete random variable \(q_t\). By Corollary~\ref{coro:mgf_q}, its MGF \(\Phi_{q_t}\) is bounded and of finite variation on \([-\pi,\pi]\). To rewrite the Fourier series in \eqref{eq:moment_q_t_series}, introduce an Abel damping parameter \(r\in(0,1)\) and set $z_r:=r e^{-\mathrm{i}u_1}.$ 	Since \(|z_r|=r<1\), the series \(\sum_{k=1}^{\infty}k^p z_r^k\) is absolutely convergent. For every integer \(p\geq0\), the classical Eulerian-polynomial identity gives
	\[
	\sum_{k=1}^{\infty} k^p z^k
	=
	\frac{zA_p(z)}{(1-z)^{p+1}},
	\qquad |z|<1,
	\]
	where \(A_0(z)=1\), and for \(p\geq1\),
	\[
	A_p(z)
	=
	\sum_{m=0}^{p-1}
	\left\langle {p \atop m} \right\rangle z^m,
	\qquad
	\left\langle {p \atop m} \right\rangle
	=
	\sum_{\ell=0}^{m+1}
	(-1)^\ell
	\binom{p+1}{\ell}
	(m+1-\ell)^p .
	\]
	Therefore,
	\begin{align}
		\sum_{k=q_*+1}^{\infty} k^p r^k e^{-\mathrm{i}u_1k}
		=
		\sum_{k=q_*+1}^{\infty} k^p z_r^k  %\nonumber\\
		%&=	\sum_{k=1}^{\infty} k^p z_r^k-sum_{k=1}^{q_*} k^pz_r^k   \nonumber\\
		=\frac{z_rA_p(z_r)}{(1-z_r)^{p+1}}-\sum_{k=1}^{q_*} k^p z_r^k .
		\label{eq:truncated_series_eulerian_abel}
	\end{align}
	For fixed \(r\in(0,1)\), this damped series is absolutely and uniformly
	convergent in \(u_1\in[-\pi,\pi]\). Indeed,
	\[
	\left|k^p r^k e^{-\mathrm{i}u_1k}\right|
	=
	k^p r^k,
	\qquad
	\sum_{k=q_*+1}^{\infty}k^p r^k<\infty,
	\]
	because the ratio test gives
	\[
	\lim_{k\to\infty}
	\frac{(k+1)^p r^{k+1}}{k^p r^k}
	=
	r<1.
	\]
	Thus the Fourier series in \eqref{eq:moment_q_t_series} is interpreted in the Abel
	sense:
	\[
	\sum_{k=q_*+1}^{\infty} k^p e^{-\mathrm{i}u_1k}
	:=
	\lim_{r\uparrow1}
	\sum_{k=q_*+1}^{\infty} k^p r^k e^{-\mathrm{i}u_1k}.
	\]
	Using \eqref{eq:truncated_series_eulerian_abel}, for every \(u_1\neq0\),
	\begin{align}
		\lim_{r\uparrow1}
		\sum_{k=q_*+1}^{\infty} k^p r^k e^{-\mathrm{i}u_1k}
		&=
		\frac{e^{-\mathrm{i}u_1}A_p(e^{-\mathrm{i}u_1})}
		{\left(1-e^{-\mathrm{i}u_1}\right)^{p+1}}
		-
		\sum_{k=1}^{q_*} k^p e^{-\mathrm{i}u_1k}.
		\label{eq:truncated_series_eulerian_boundary}
	\end{align}
	The only singularity is at \(u_1=0\), corresponding to the divergence
	of \(\sum_{k=q_*+1}^{\infty}k^p\). 
	%Hence the expression is understood through	the Abel limit, and 
	Since the singular point has Lebesgue measure zero in
	\([-\pi,\pi]\), the integral in \eqref{eq:moment_q_t} is well defined.
	%in this Abel-regularised sense.
	
	We now treat the two continuous random variables \(\lambda_t\) and \(\bar q_t\) simultaneously. Let \(X\) denote either \(\lambda_t\) or \(\bar q_t\), let \(x_*\) denote the corresponding threshold, and let $\Phi_X(u):=\mathbb{E}[e^{uX}]$ be the associated MGF. Thus
\[
(X,x_*,\Phi_X)
=
(\lambda_t,\lambda_*,\Phi_{\lambda_t})
\quad\text{or}\quad
(\bar q_t,\bar q_*,\Phi_{\bar q_t}).
\]
By the binomial expansion
\begin{align}
	\mathbb{E}\!\left[X^p\mathbf{1}_{\{X>x_*\}}\right]
	&=
	\mathbb{E}\!\left[(X-x_*+x_*)^p\mathbf{1}_{\{X>x_*\}}\right]  = \sum_{j=0}^{p} \binom{p}{j}x_*^{p-j} \mathbb{E}\!\left[(X-x_*)_+^j\right]. 	\label{eq:binomial_tail_decomposition}
\end{align}
Set \(V:=X-x_*\). For \(\alpha_1>0\), the Fourier--Laplace representation
\citep[Eq.~3.2(3)]{BatemanEderlyi1954} gives, for every \(j=0,\ldots,p\), that 
\[
v_+^j
=
\frac{\Gamma(j+1)e^{\alpha_1 v}}{\pi}
\Re\left(
\int_0^{+\infty}
\frac{e^{-\mathrm{i}u_1 v}}
{(\alpha_1-\mathrm{i}u_1)^{j+1}}
\,du_1
\right).
\]
Taking expectations and applying Fubini's theorem yields:
\begin{align}
	\mathbb{E}[V_+^j]
	&=
	\frac{\Gamma(j+1)}{\pi}
	\int_0^{+\infty}
	\Re\left(
	\frac{\mathbb{E}\!\left[
		e^{(\alpha_1-\mathrm{i}u_1)V}
		\right]}
	{(\alpha_1-\mathrm{i}u_1)^{j+1}}
	\right)\,du_1                                      \nonumber\\
	&=
	\frac{\Gamma(j+1)}{\pi}
	\int_0^{+\infty}
	\Re\left(
	\frac{
		e^{-(\alpha_1-\mathrm{i}u_1)x_*}
		\Phi_X(\alpha_1-\mathrm{i}u_1)}
	{(\alpha_1-\mathrm{i}u_1)^{j+1}}
	\right)\,du_1 .
	\label{eq:positive_part_transform}
\end{align}
Substituting \eqref{eq:positive_part_transform} into
\eqref{eq:binomial_tail_decomposition} gives the general formula
\begin{align}
	\mathbb{E}\!\left[X^p\mathbf{1}_{\{X>x_*\}}\right]
	&=
	\int_0^{+\infty}
	\sum_{j=0}^{p}
	\binom{p}{j}
	\frac{x_*^{p-j}\Gamma(j+1)}{\pi}
	\Re\left(
	\frac{
		e^{-(\alpha_1-\mathrm{i}u_1)x_*}
		\Phi_X(\alpha_1-\mathrm{i}u_1)}
	{(\alpha_1-\mathrm{i}u_1)^{j+1}}
	\right)\,du_1 .
	\label{eq:continuous_tail_moment_general}
\end{align}
Choosing \(X=\lambda_t\) and \(x_*=\lambda_*\) gives \eqref{eq:moment_lambda_t}, while choosing \(X=\bar q_t\) and  \(x_*=\bar q_*\) gives \eqref{eq:moment_bar_q_t}.  It remains only to justify the interchange of expectation and integration in \eqref{eq:positive_part_transform}. For \(j=1,\ldots,p\), we use
\[
\left|
\frac{\Phi_X(\alpha_1-\mathrm{i}u_1)}
{(\alpha_1-\mathrm{i}u_1)^{j+1}}
\right|
\leq
\frac{\Phi_X(\alpha_1)}
{\left|\alpha_1-\mathrm{i}u_1\right|^{j+1}},
\]
which is integrable on \([0,\infty)\) because \(j+1>1\) and \(\Phi_X(\alpha_1)<\infty\). The finiteness of the MGF at \(\alpha_1\) follows from the corresponding MGF formula: Corollary~\ref{coro:mgf_lambda} for \(X=\lambda_t\), and Corollary~\ref{coro:mgf_integral_q} for \(X=\bar q_t\). For \(j=0\), the same MGF expressions imply that the map $u_1\longmapsto \Phi_X(\alpha_1-\mathrm{i}u_1)$ is bounded and of finite variation on \([0,\infty)\). Hence Dirichlet's test for oscillatory integrals ensures the convergence of
\[
\int_0^{+\infty}
\frac{\Phi_X(\alpha_1-\mathrm{i}u_1)}
{\alpha_1-\mathrm{i}u_1}\,du_1,
\]
and therefore the required Fubini interchange is justified. This completes the proof.
\end{proof}

The aforementioned result concerns risk measures for a single variable. Of particular interest is the multivariate case involving joint moments conditional on one variable exceeding a given threshold. More precisely, consider the mixed truncated moment 
\[ 
	\mathbb{E}\!\left[q_t^r \lambda_t^p \mathbf{1}_{\{\lambda_t>\lambda_*\}}\right]\,, 
\] 
which represents the joint moment of order \((r,p)\) of the pair \((q_t,\lambda_t)\), restricted to the event that the arrival intensity exceeds the threshold \(\lambda_*\). In this expression, the number in the system \(q_t\) can be replaced by \(\bar{q}_t\). Alternatively, the pairs \((q_t,\bar{q}_t)\) or \((\bar{q}_t,q_t)\) may be considered, depending on which variable is subject to truncation. The following proposition provides an analytical representation based on the joint MGF and the result of Proposition~\ref{prop:JointExpectedShortfallDerivativeFourierTransform}.

\begin{proposition}
	\label{prop:JointExpectedShortfallDerivativeFourierTransform_queue}
	
	Let $(\lambda_t,q_t)\in\mathbb{R}_+ \times \mathbb{N}$ be a random vector with joint MGF
	$\Phi_{(\lambda_t,q_t)}(v_1,v_2)$.
	Let $\Phi_{(\lambda_t,q_t)}^{(0,r)}$ denote the $r$-th derivative with
	respect to the second argument. For $\lambda_*>0$ and $\alpha_1>0$,
	\begin{align}
		\mathbb{E}[q_t^r\lambda_t^p \mid \lambda_t>\lambda_*]
		=
		\frac{\mathbb{E}[q_t^r\lambda_t^p
			\mathbf{1}_{\{\lambda_t>\lambda_*\}}]}
		{\mathbb{P}(\lambda_t>\lambda_*)},
	\end{align}
	where
	\begin{align}
		\mathbb{E}[q_t^r \lambda_t^p
		\mathbf{1}_{\{\lambda_t>\lambda_*\}}]
		&=		\int_{0}^{+\infty}
		\sum_{j=0}^{p}
		\binom{p}{j}
		\frac{\lambda_*^{p-j}\Gamma(j+1)}{\pi}
		\Re\!\left(
		\frac{e^{-(\alpha_1-\mathrm{i}u_1)\lambda_*}
			\Phi_{(\lambda_t,q_t)}^{(0,r)}(\alpha_1-\mathrm{i}u_1,0)}
		{(\alpha_1-\mathrm{i}u_1)^{j+1}}
		\right)du_1 .
		\label{eq:ExpectationZqYp-1dim_queue}
	\end{align}
	
		Let $(q_t,\lambda_t)\in \mathbb{N} \times \mathbb{R}_+ $ be a random vector with joint MGF 	$\Phi_{(q_t,\lambda_t)}(v_1,v_2)$. 	Let $\Phi_{(q_t,\lambda_t)}^{(0,r)}$ denote the $r$-th derivative with 	respect to the second argument. For $q_*>0$
	\begin{align}
		\mathbb{E}[\lambda_t^r q_t^p\mid q_t>q_*]
		=
		\frac{\mathbb{E}[\lambda_t^r q_t^p \mathbf{1}_{\{q_t>q_*\}}]}
		{\mathbb{P}(q_t>q_*)},
	\end{align}
	where
	\begin{align}
		\mathbb{E}[\lambda_t^r q_t^p \mathbf{1}_{\{q_t>q_*\}}]
		&=		\frac{1}{2\pi}\int_{-\pi}^{+\pi} \Re \left( \Phi_{(q_t,\lambda_t)}^{(0,r)}(\mathrm{i}u_1,0)\left(\sum_{k=q_*+1}^{+\infty}k^pe^{- \mathrm{i} u_1 k}\right)\right)du_1\,.\label{eq:ExpectationZqYp-1dim_queueB}
	\end{align}
		
	Let $(\lambda_t,\bar{q}_t)\in \mathbb{R}_+ \times \mathbb{R}_+ $ be a random vector with joint MGF 	$\Phi_{(\lambda_t,\bar{q}_t)}(v_1,v_2)$. 	Let $\Phi_{(\lambda_t,\bar{q}_t)}^{(0,r)}$ denote the $r$-th derivative with respect to the second argument. For $\bar{q}_*>0$ and $\alpha_1>0$,
	\begin{align}
		\mathbb{E}[\bar{q}_t^r \lambda_t^p \mid \lambda_t>\lambda_*]
		=
		\frac{\mathbb{E}[\bar{q}_t^r \lambda_t^p \mathbf{1}_{\{\lambda_t>\lambda_*\}}]}
		{\mathbb{P}(\lambda_t>\lambda_*)},
	\end{align}
	where
	\begin{align}
		\mathbb{E}[\bar{q}_t^r \lambda_t^p \mathbf{1}_{\{\lambda_t>\lambda_*\}}]
		&=		\int_{0}^{+\infty}
		\sum_{j=0}^{p}
		\binom{p}{j}
		\frac{\lambda_*^{p-j}\Gamma(j+1)}{\pi}
		\Re\!\left(
		\frac{e^{-(\alpha_1-\mathrm{i}u_1)\lambda_*}
			\Phi_{(\lambda_t,\bar{q}_t)}^{(0,r)}(\alpha_1-\mathrm{i}u_1,0)}
		{(\alpha_1-\mathrm{i}u_1)^{j+1}}
		\right)du_1 .\label{eq:Expectation_bar_qt_r_lambdat_p-1dim}
	\end{align}
	
\end{proposition}

\begin{proof}
	Proceeding as in the previous proof, we obtain
	\begin{align*}
		\mathbb{E}[q_t^r \lambda_t^p \mathbf{1}_{\lbrace \lambda_t>\lambda_* \rbrace}]		&=\sum_{j=0}^p
		\binom{p}{j}
		\lambda_*^{p-j}
		\mathbb{E}[q_t^{r}(\lambda_t-\lambda_*)_+^j],
	\end{align*}
	where
	\begin{align*}
		\mathbb{E}\!\left[q_t^r(\lambda_t-\lambda_*)_+^j\right]
		= \frac{\Gamma(j+1)}{2\pi}
		\int_{-\infty}^{+\infty}
		\Re \left(
		\frac{e^{-(\alpha_1-\mathrm{i}u_1)\lambda_*}
			\mathbb{E}\!\left[q_t^r e^{(\alpha_1-\mathrm{i}u_1)\lambda_t}\right]}
		{(\alpha_1-\mathrm{i}u_1)^{j+1}}
		\right)
		\,du_1,
	\end{align*}
	for $\alpha_1>0$. Moreover,
	\[
	\mathbb{E}\!\left[q_t^r e^{(\alpha_1-\mathrm{i}u_1)\lambda_t}\right]
	= \left.\frac{d^r}{dv_2^r}
	\mathbb{E}\!\left[e^{(\alpha_1-\mathrm{i}u_1)\lambda_t+v_2 q_t}\right]\right|_{v_2=0}
	= \Phi_{(\lambda_t,q_t)}^{(0,r)}(\alpha_1-\mathrm{i}u_1,0).
	\]
	Substituting this expression into the integral yields the desired result. The application of Fubini's theorem is justified by the integrability of the mapping
	\[
	u_1 \mapsto \frac{\Phi_{(\lambda_t,q_t)}^{(0,r)}(\alpha_1-\mathrm{i}u_1,0)}
	{(\alpha_1-\mathrm{i}u_1)^{j+1}}, \qquad j=1,\ldots,p,
	\]
	which is lenghty to check but feasible.	Similar arguments apply to \eqref{eq:Expectation_bar_qt_r_lambdat_p-1dim}.
\end{proof}

\subsection{Evaluation of risk measures of infinite-server queues at two dates}\label{sec:QueueRiskMeasuresTwoDates}

The time-dependent structure of the Cox--BESQ process allows us to study risk measures at two distinct dates. Let
$0<t_1<t_2$. We focus on time-lagged conditional cross moments of the form 
\[
\mathbb{E}\!\left[q_{t_2}^r q_{t_1}^p \mid q_{t_1}>q_*\right]=\frac{\mathbb{E}\!\left[q_{t_2}^r q_{t_1}^p \mathbf{1}_{\{q_{t_1}>q_*\}} \right]}{\mathbb{E}\!\left[\mathbf{1}_{\{q_{t_1}>q_*\}}\right]} ,\qquad
\mathbb{E}\!\left[\lambda_{t_2}^r q_{t_2}^p \mid \lambda_{t_1}>\lambda_*\right] = \frac{\mathbb{E}\!\left[\lambda_{t_2}^r q_{t_2}^p \mathbf{1}_{\lbrace \lambda_{t_1}>\lambda_*\}} \right]}{\mathbb{E}\!\left[\mathbf{1}_{\lbrace \lambda_{t_1}>\lambda_*\}}\right]},
\]
where $q_*>0$, $\lambda_*>0$, and $r,q\in\mathbb{N}$. We restrict the results to the triplet of variables \(\lambda_{t_1}, q_{t_1}, q_{t_2}\), but it should be clear that any combination of the variables \((\lambda_{t_1},\lambda_{t_2}, q_{t_1}, q_{t_2}, \bar{q}_{t_1}, \bar{q}_{t_2})\) could be considered.

These quantities measure the persistence of extreme number of customers in the system lengths or intensities across time and are related to the time-series literature on risk measures; see, for instance,
\citet{ChavezDemoulinGuillou2018} and \citet{GoegebeurGuillou2024}. By combining the two-date MGFs of Corollary~\ref{coro:qt_lambdat_two_dates} with Proposition~\ref{prop:JointExpectedShortfallDerivativeFourierTransform}, we obtain the following explicit representations for the two aforementioned risk measures, as shown in the following proposition.

\begin{proposition}
	\label{prop:ConditionalMomentsTwoDates}
	Let $t_1,t_2\in\mathbb{R}_+$ satisfying $0<t_1<t_2$, and let $\alpha_1>0$ be an admissible damping parameter.
	
	\begin{enumerate}
		\item Let \((q_{t_1},q_{t_2})\) the number of customers in the system at the dates \(t_1,\, t_2\) and \(\Phi_{(q_{t_1},q_{t_2})}(u_1,u_2)=\mathbb{E}\!\left[e^{u_1q_{t_1}+u_2q_{t_2}} \right]\) the MGF of this vector that is given by \eqref{eq:mgf_q_t1_t2} whilst the \(r\)-derivative of this function with respect to its second argument is denoted \(\Phi_{(q_{t_1},q_{t_2})}^{(0,r)}\). For $q_*>0$, the $(r,p)$ cross moment of $q_{t_2}$ and $q_{t_1}$  conditional on \(q_{t_1}>q_*\) is:
		\begin{align}
			\mathbb{E}\!\left[q_{t_2}^r q_{t_1}^p \mathbf{1}_{\{q_{t_1}>q_*\}}\right]
			&=
			\frac{1}{2\pi}\int_{-\pi}^{+\pi} \Re \left( \Phi_{(q_{t_1},q_{t_2})}^{(0,r)}(\mathrm{i}u_1,0)\left(\sum_{k=q_*+1}^{+\infty}k^pe^{- \mathrm{i} u_1 k}\right)\right)du_1\,.
			\label{eq:conditional_moment_q_two_dates}
		\end{align}
		
		\item Let \((\lambda_{t_1},\lambda_{t_2},q_{t_2})\) the vector of the arrival intensity at dates \(t_1,\, t_2\) and the number of customers in the system at the date \(t_2\) and \(\Phi_{(\lambda_{t_1},\lambda_{t_2},q_{t_2})}(u_1,u_2,u_3)=\mathbb{E}\!\left[ e^{u_1\lambda_{t_1} + u_2\lambda_{t_2} + u_3q_{t_2}}\right]\) the corresponding MGF given by \eqref{eq:mgf_lambda_t1_lambda_t2_q_t2}, and its \(r,p\)-derivative with respect to \(u_2,\,u_3\) is denoted \(\Phi_{(\lambda_{t_1},\lambda_{t_2},q_{t_2})}^{(0,r,p)}\).  For $\lambda_*>0$, the $r,p$-th conditional cross moment of $\lambda_{t_2},\, q_{t_2}$ given $\lambda_{t_1}>\lambda_*$ is:
		\begin{align}
			\mathbb{E}\!\left[ \lambda_{t_2}^r q_{t_2}^p \mathbf{1}_{\{ \lambda_{t_1}>\lambda_*\}}\right]
			&=
				\frac{1}{\pi}\int_0^{+\infty}
				\Re\!\left(
				\frac{ e^{-(\alpha_1-\mathrm{i}u_1)\lambda_*} \Phi_{(\lambda_{t_1},\lambda_{t_2},q_{t_2})}^{(0,r,p)}
					(\alpha_1-\mathrm{i}u_1,0,0)
				}{
					\alpha_1-\mathrm{i}u_1
				}
				\right)\,du_1
			.
			\label{eq:conditional_moment_lambda_two_dates}
		\end{align}
		
	\end{enumerate}
\end{proposition}

\section{Numerical Experiments}\label{sec:numerical_experiments}

In this section, we present numerical experiments illustrating the practical implementation of the analytical formulas derived in previous sections. The objective is twofold. First, we show how the proposed transform-based methodology can be used to evaluate tail-related risk measures associated with the number of customres in the system, its time integral or the arrival intensity. Secondly, we demonstrate that the same framework naturally extends to joint and cross-risk quantities involving these variables.

A key advantage of the proposed approach is that it relies directly on the MGF of the relevant random variable rather than on a closed-form expression for its density or probability mass function. This is particularly useful in queueing models with stochastic intensity where explicit distributional formulas are often unavailable. Fourier inversion then provides an effective numerical tool for evaluating probabilities, truncated expectations, and conditional tail quantities such as expected shortfall.

Throughout this section, all numerical values are computed using the parameter specification reported in Table~\ref{tab:parameters}. Unless otherwise stated, the Fourier inversion is performed over the truncated integration domain $[0,u_{\max}]$, with $u_{\max}=100$, and the damping parameter is chosen to ensure numerical stability.
\begin{center}
\begin{table}[ht]
	\caption{Parameter values used in the numerical experiments.}
	\label{tab:parameters}
	\[
	\begin{array}{ccccccccccccc}
		\hline
		\text{Parameter}   & \kappa 	& \theta 	& \sigma 	& \mu & \lambda_0 & q_0  & t,\, t_1 & t_2 & \lambda_* & q_* & \bar{q}_* & \alpha_1 \\ 	\hline
		\text{Value} 			& 0.7 		& 2  			& 0.3 		& 1  	& 2  				& 4  	 & 1 				& 2 	&  2.5 				& 5 		& 		8				&  1.2			 \\ 		\hline
	\end{array}
	\]
\end{table}
\end{center}
\subsection{Computation of tail risk measures}
\label{subsec:TailRiskMeasuresNumerics}
\paragraph{Risk measures for \(\lambda_t\):}

	We first consider the expected shortfall associated with intensity $\lambda_t$. In order to apply Proposition~\ref{prop:qt_ExpectedShortfallFourierTransform}, we specify a threshold $\lambda_*$ representing the upper tail of the distribution. To this end, we numerically recover the probability mass function of $\lambda_t$ from its MGF and inspect the resulting distribution. Note that in this particular case, then density of the intensity is known explicitly, see \citet[Eq.~(6.2.2)]{JeanblancYorChesney2009}, but we still use its MGF to determine numerically the density. The probability mass function is displayed in Figure~\ref{fig:lambda_pdf}.The plot indicates that
\[
\mathbb{P}(\lambda_t \leq 2.5)\approx 0.94,
\]
so that $\lambda_*=2.5$ is a natural threshold for the upper-tail event. The recovered distribution indicates that

\[\mathbb{E}\!\left[
\mathbf{1}_{\{\lambda_t>\lambda_*\}}
\right]
=
\mathbb{P}(\lambda_t>\lambda_*)
\approx 0.061.\]
Using \eqref{eq:moment_lambda_t} with \(p=1\) and \(p=2\) we obtain 
\[
\mathbb{E}\!\left[\lambda_t\,\mathbf{1}_{\{\lambda_t>\lambda_*\}}\right]\approx 0.162,
\qquad
\mathbb{E}\!\left[\lambda_t^2\,\mathbf{1}_{\{\lambda_t>\lambda_*\}}\right]\approx 0.433,
\]
from which we deduce that the expected shortfall of the intensity above the threshold $\lambda_*$ is
\[
\mathbb{E}\!\left[\lambda_t \mid \lambda_t>\lambda_*\right]
=
\frac{
	\mathbb{E}\!\left[\lambda_t\,\mathbf{1}_{\{\lambda_t>\lambda_*\}}\right]
}{
	\mathbb{P}(\lambda_t>\lambda_*)
}
\approx 2.658,
\]
while the tail variance is
\begin{align}
	\TV_{\lambda_t}(\lambda_*)
=
\mathbb{E}[\lambda_t^2\mid \lambda_t>\lambda_*]
-
\left(\mathbb{E}[\lambda_t\mid \lambda_t>\lambda_*]\right)^2
=\frac{\mathbb{E}\!\left[\lambda_t^2\,\mathbf{1}_{\{\lambda_t>\lambda_*\}}\right]
}{\mathbb{P}(\lambda_t>\lambda_*)}
-\left(\mathbb{E}[\lambda_t\mid \lambda_t>\lambda_*]\right)^2\approx 0.021.
\end{align}

\paragraph{Risk measures for \(q_t\):} We next consider the expected shortfall associated with the number in the system $q_t$. As before, we specify a threshold $q_*$ representing the upper tail of the distribution and we numerically recover the probability mass function of $q_t$ from its MGF and inspect the resulting distribution. The probability mass function is displayed in Figure~\ref{fig:q_pmf}.

The recovered distribution indicates that
\[
\mathbb{P}(q_t\leq q_*=5)\approx 0.96.
\]
Thus, the level $q_*=5$ provides an upper-tail threshold close to the $96\%$ quantile. Using the inversion formula of~\eqref{eq:moment_q_t}, we obtain
\[
\mathbb{P}(q_t>q_*)\approx 0.041,
\qquad
\mathbb{E}\!\left[q_t\,\mathbf{1}_{\{q_t>q_*\}}\right]\approx 0.262,
\qquad
\mathbb{E}\!\left[q_t^2\,\mathbf{1}_{\{q_t>q_*\}}\right]\approx 1.686,
\]
and, consequently, the expected shortfall of the number in the system above the threshold $q_*$ is
\[
\mathbb{E}\!\left[q_t \mid q_t>q_*\right]
=
\frac{
	\mathbb{E}\!\left[q_t\,\mathbf{1}_{\{q_t>q_*\}}\right]
}{
	\mathbb{P}(q_t>q_*)
}
\approx 6.375,
\]
while the tail variance is 
\begin{align}
	\TV_{q_t}(q_*)
	=
	\mathbb{E}[q_t^2\mid q_t>q_*]
	-
	\left(\mathbb{E}[q_t\mid q_t>q_*]\right)^2
	=\frac{\mathbb{E}\!\left[q_t^2\,\mathbf{1}_{\{q_t>q_*\}}\right]
	}{\mathbb{P}(q_t>q_*)}
	-\left(\mathbb{E}[q_t\mid q_t>q_*]\right)^2\approx 0.459.
\end{align}

These values represent the average and variance of number of customers in the system conditional its value being in its upper tail.

\begin{figure}[h!] 
	\centering 
	\begin{subfigure}[b]{0.48\textwidth} 
		\centering \includegraphics[width=\textwidth,height=5cm]{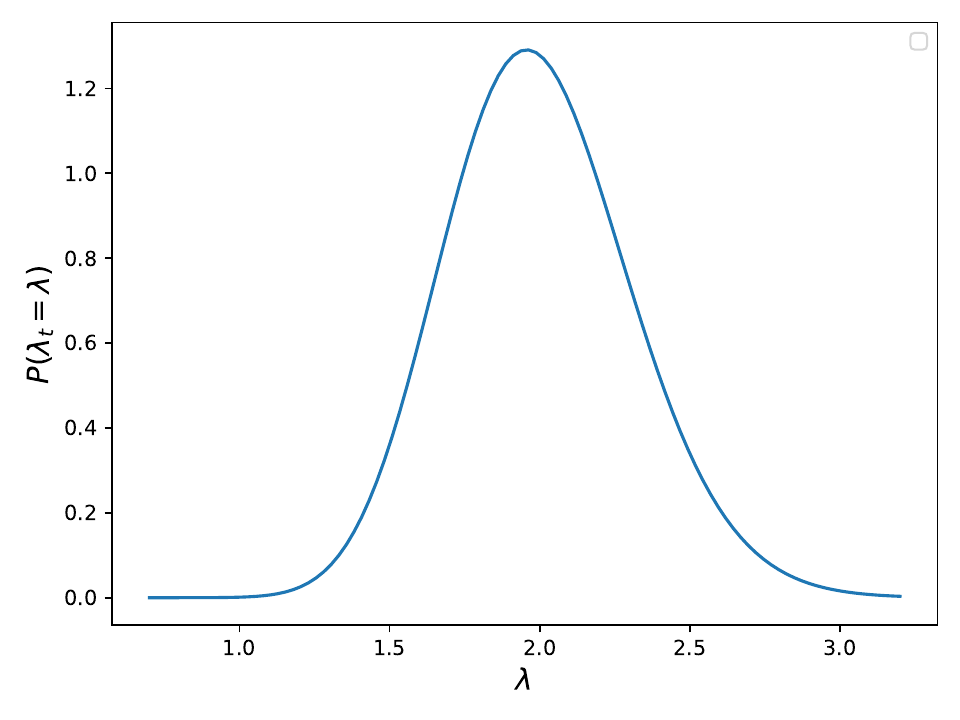} 
		\caption{Panel A} 
		\label{fig:lambda_pdf} 
	\end{subfigure} 
	\hfill 
	\begin{subfigure}[b]{0.48\textwidth} 
		\centering \includegraphics[width=\textwidth,height=5cm]{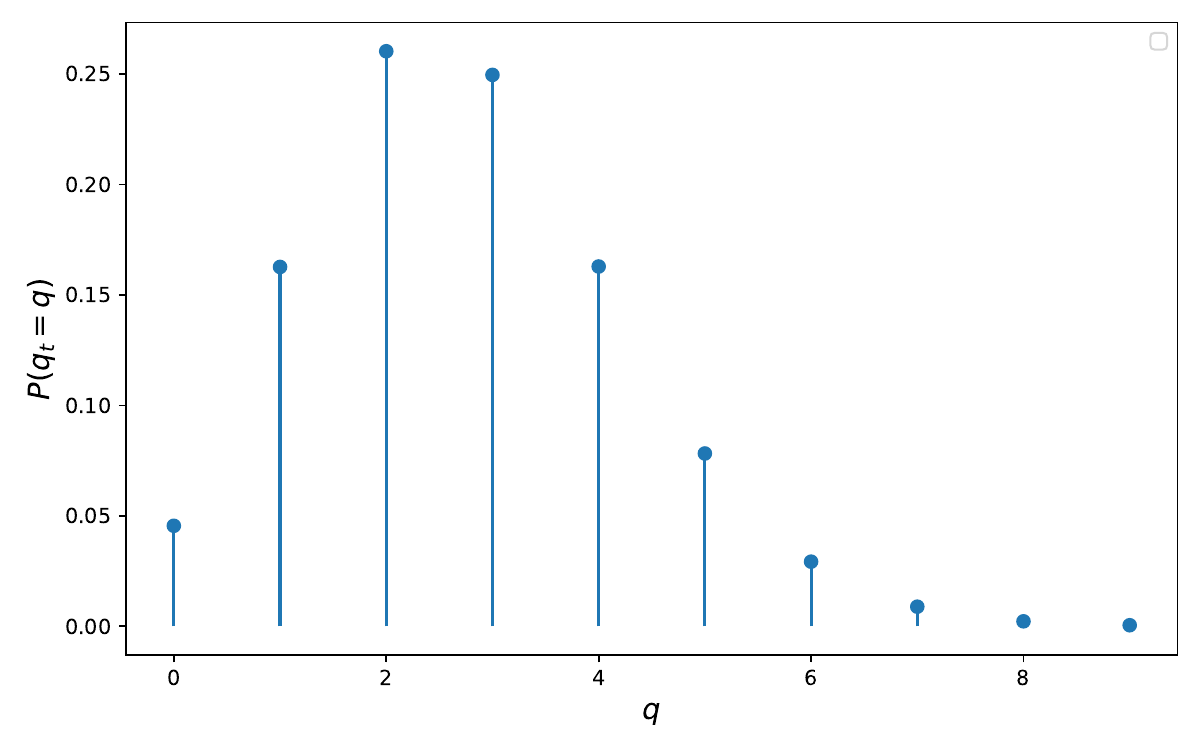} 
		\caption{Panel B} 
		\label{fig:q_pmf} 
	\end{subfigure} 
	\caption{Probability density function of $\lambda_t$ (left panel) and probability mass function of $q_t$ (right panel).} 
	\label{fig:two_figures} 
\end{figure}

\paragraph{Risk measure for \(\bar{q}_t\):}  We next consider the cumulative infinite-server queues occupancy over the time interval $[0,t]$, namely \(\bar{q}_t\). As before, the first step is to identify a suitable tail threshold. This is done by numerically inverting the corresponding transform and plotting the probability density function for \(\bar{q}_t\) and find the adequate level for $\bar q_*$ that gives the targeted quantile.

Numerically we find that the level $\bar q_*=8.4$ is a reasonable approximation to a $91\%$ tail threshold for $\bar{q}_t$. To compute the corresponding tail expectation, we use the same inversion formula as in~\eqref{eq:mgf_integral_q}, replacing the MGF of $q_t$ by the MGF of $\bar{q}_t$ given in Corollary~\ref{coro:mgf_integral_q}. Using \eqref{eq:moment_bar_q_t} for \(p=0\), \(p=1\) and \(p=2\) yields
\[
\mathbb{P}\!\left(\bar{q}_t>\bar q_*\right)
\approx 0.090,
\qquad 
\mathbb{E}\!\left[
\bar{q}_t
\mathbf{1}_{\left\{\bar{q}_t>\bar q_*\right\}}
\right]
\approx 0.834,
\qquad
\mathbb{E}\!\left[
\bar{q}_t^2
\mathbf{1}_{\left\{\bar{q}_t>\bar q_*\right\}}
\right]
\approx 7.845,
\]

so that the expected shortfall for \(\bar{q}_t\) above the threshold $\bar{q}_*$ is
\[
\mathbb{E}\!\left[
\bar{q}_t
\,\middle|\,
\bar{q}_t>\bar{q}_*
\right]
=
\frac{
	\mathbb{E}\!\left[
	\bar{q}_t
	\mathbf{1}_{\left\{\bar{q}_t>\bar{q}_*\right\}}
	\right]
}{
	\mathbb{P}\!\left(\bar{q}_t>\bar{q}_*\right)
}
\approx 9.27.
\]
while the corresponding tail variance is 
\begin{align}
	\TV_{\bar{q}_t}(\bar q_*)
	=\frac{
			\mathbb{E}\!\left[
			\bar{q}_t^2
			\mathbf{1}_{\left\{\bar{q}_t>\bar{q}_*\right\}}
			\right]
		}{
			\mathbb{P}\!\left(\bar{q}_t>\bar{q}_*\right)
	}
	-\mathbb{E}\!\left[
	\bar{q}_t
	\,\middle|\,
	\bar{q}_t>\bar{q}_*
	\right]^2\approx 1.233.
\end{align}

These three  examples illustrate the usefulness of the transform approach in two distinct settings: a discrete state variable, namely the terminal value of the number in the system, and a continuous path-dependent quantity, namely the cumulative value over time of the number in the sytem or the arrival intensity. In both cases, the methodology is the same: identify a suitable tail threshold from the recovered distribution, compute the tail probability and the truncated first moment, and then form the associated conditional tail expectation.

\subsection{Cross-risk measures involving the stochastic intensity} \label{subsec:CrossRiskMeasuresNumerics}

We now turn to joint quantities involving both the number of customers in the system process \(q_t\) and the stochastic intensity process \(\lambda_t\). The objective is to evaluate cross-risk measures that capture the interaction between these two quantities conditional on one of them being unusually high. To implement Proposition~\ref{prop:JointExpectedShortfallDerivativeFourierTransform_queue}, we take the threshold $\lambda_*$ for the intensity found above. Then, using equation~\eqref{eq:ExpectationZqYp-1dim_queue}, we compute
\[
\mathbb{E}\!\left[
q_t\lambda_t\,\mathbf{1}_{\{\lambda_t>\lambda_*\}}
\right]\approx 0.482,
\]
and therefore the corresponding conditional cross-risk measure is
\[
\mathbb{E}\!\left[q_t\lambda_t \mid \lambda_t>\lambda_*\right]
=
\frac{
	\mathbb{E}\!\left[
	q_t\lambda_t\,\mathbf{1}_{\{\lambda_t>\lambda_*\}}
	\right]
}{
	\mathbb{P}(\lambda_t>\lambda_*)
}
\approx 7.894.
\]
Similarly,  using equation~\eqref{eq:ExpectationZqYp-1dim_queueB}, we obtain
\[
\mathbb{E}\!\left[
q_t\lambda_t\,\mathbf{1}_{\{q_t>q_*\}}
\right]\approx 0.541,
\]
which yields to
\[
\mathbb{E}\!\left[q_t\lambda_t \mid q_t>q_*\right]
=
\frac{
	\mathbb{E}\!\left[
	q_t\lambda_t\,\mathbf{1}_{\{q_t>q_*\}}
	\right]
}{
	\mathbb{P}(q_t>q_*)
}
\approx 13.186.
\]

The first quantity measures the average value of the product \(q_t\lambda_t\) in scenarios where the intensity process is in its upper tail, whereas the second corresponds to scenarios in which the number of customers in the system is in its upper tail.

A similar calculation can be performed for the time-integrated number in the system \(\bar{q}_t\). Applying again~\eqref{eq:ExpectationZqYp-1dim_queue}, now with \(\bar{q}_t\) in place of $q_t$, we obtain
\[
\mathbb{E}\!\left[
\bar{q}_t
\lambda_t\,\mathbf{1}_{\{\lambda_t>\lambda_*\}}
\right]\approx 0.919,
\]
and the cross-moment between \(\bar{q}_t\) and \(\lambda_t\) conditional on the intensity being above the threshold \(\lambda_*\) is
\[
\mathbb{E}\!\left[
\bar{q}_t\lambda_t
\,\middle|\,
\lambda_t>\lambda_*
\right]
=
\frac{
	\mathbb{E}\!\left[
	\bar{q}_t
	\lambda_t\,\mathbf{1}_{\{\lambda_t>\lambda_*\}}
	\right]
}{
	\mathbb{P}(\lambda_t>\lambda_*)
}
\approx 15.049.
\]
Compared with the previous conditional moment involving \(q_t\lambda_t\), this larger value reflects the cumulative nature of \(\bar{q}_t\), which aggregates \(q_t\) over the entire time horizon rather than at a single terminal date. As a consequence, it captures prolonged episodes of system stress and may provide a more informative measure in applications where persistently high levels of \(q_t\) are of particular concern.

\subsection{Two-date risk measures}\label{subsec:TwoDateRiskMeasuresNumerics}

We now illustrate the computation of risk measures involving two observation dates. The time-dependent structure of the Cox--BESQ model allows us to evaluate conditional moments at a future date \(t_2\), given that a tail event occurs at an earlier date \(t_1\), with \(0 < t_1 < t_2\). One could reverse the conditioning, that is, compute the expectation of the variables at time \(t_1\) conditional on a variable exceeding a threshold at the future date \(t_2\), but this appears to be of less practical interest. In this subsection, we consider
\[
\mathbb{E}\!\left[\lambda_{t_2}\mid \lambda_{t_1}>\lambda_*\right],
\qquad
\mathbb{E}\!\left[q_{t_2}\mid q_{t_1}>q_*\right],
\qquad
\mathbb{E}\!\left[q_{t_2}\lambda_{t_2}\mid \lambda_{t_1}>\lambda_*\right],
\qquad
\mathbb{E}\!\left[q_{t_2}\lambda_{t_2}\mid q_{t_1}>q_*\right].
\]
These quantities quantify the persistence of the number of customers in the system and the arrival intensity across time.

Using Proposition~\ref{prop:ConditionalMomentsTwoDates} with $\lambda_*=2.5$, we find
\[
\mathbb{E}\!\left[
\lambda_{t_2}\mathbf{1}_{\{\lambda_{t_1}>\lambda_*\}}
\right]\approx 0.142.
\]
Hence,
\[
\mathbb{E}\!\left[\lambda_{t_2}\mid \lambda_{t_1}>\lambda_*\right]
=
\frac{
	\mathbb{E}\!\left[
	\lambda_{t_2}\mathbf{1}_{\{\lambda_{t_1}>\lambda_*\}}
	\right]
}{
	\mathbb{P}(\lambda_{t_1}>\lambda_*)
}
\approx 2.327.
\]
This conditional moment measures the persistence of high arrival intensity from time $t_1$ to the later date $t_2$.

Figure~\ref{fig:lambdat_2dates} displays the dependence of $\mathbb{E}\!\left[\lambda_{t_2}\mid \lambda_{t_1}>\lambda_*\right]$ on $t_2$ over the range $t_2 \in [1,8]$ (knowing that \(t_1=1\)). The curve decays towards its long-term mean value \(\theta=\mathbb{E}[\lambda_{\infty}]\), reflecting the mean-reverting behaviour of the stochastic intensity process.

Next, using equation~\eqref{eq:conditional_moment_q_two_dates}, with the same number of customers in the system threshold $q_*$, we compute
\[
\mathbb{E}\!\left[
q_{t_2}\mathbf{1}_{\{q_{t_1}>q_*\}}
\right]\approx 0.149,
\]
and it follows that
\[
\mathbb{E}\!\left[q_{t_2}\mid q_{t_1}>q_*\right]
=
\frac{
	\mathbb{E}\!\left[
	q_{t_2}\mathbf{1}_{\{q_{t_1}>q_*\}}
	\right]
}{
	\mathbb{P}(q_{t_1}>q_*)
}
\approx 3.634.
\]
This value gives the expected number in the system at the later time $t_2$, conditional on its value being unusually large at time $t_1$. Figure~\ref{fig:qt_2dates} displays the dependence of the two-date expected value
$\mathbb{E}\!\left[q_{t_2}\mid q_{t_1}>q_*\right]$ on $t_2$ over the range $t_2 \in [1,8]$. The curve exhibits mean-reverting behaviour and decays towards its long-term mean value $\frac{\theta}{\mu}$ given in \eqref{Eqn:q_long_term}.

\begin{figure}[h!] 
	\centering 
	\begin{subfigure}[b]{0.48\textwidth} 
		\centering \includegraphics[width=\textwidth,height=5cm]{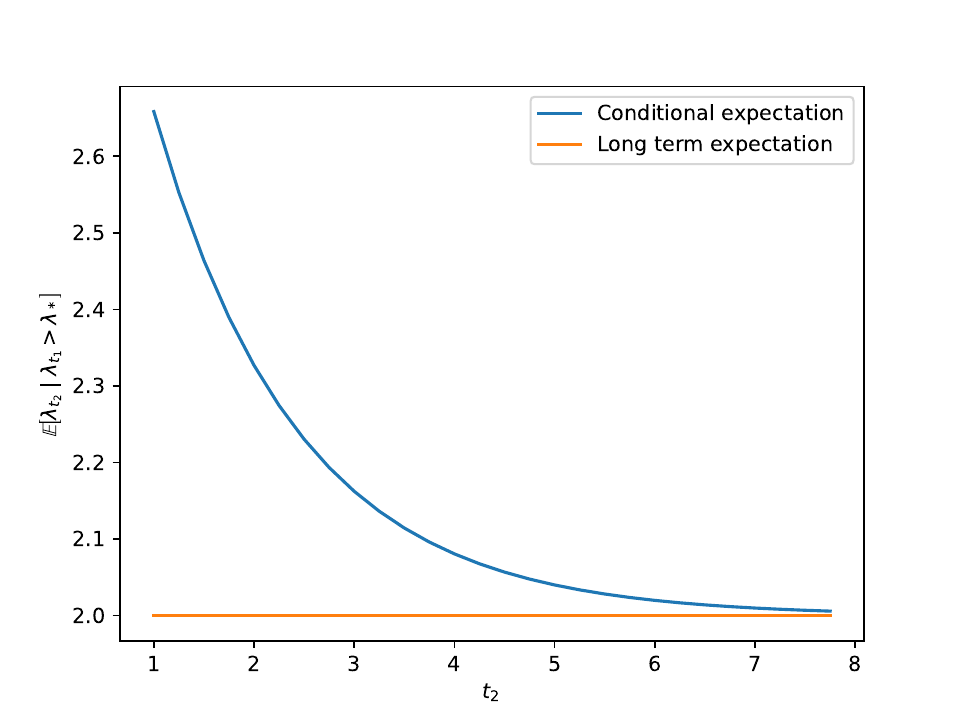} 
		\caption{Panel A} 
		\label{fig:lambdat_2dates} 
	\end{subfigure} 
	\hfill 
	\begin{subfigure}[b]{0.48\textwidth} 
		\centering \includegraphics[width=\textwidth,height=5cm]{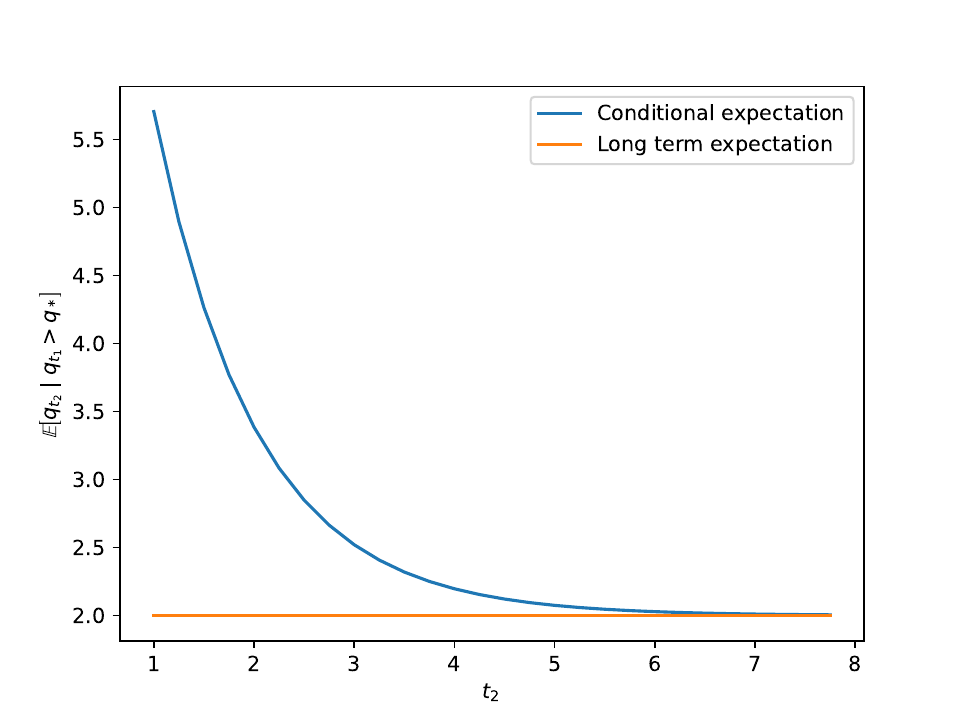} 
		\caption{Panel B} 
		\label{fig:qt_2dates} 
	\end{subfigure} 
	\caption{Dependence of $\mathbb{E}\!\left[\lambda_{t_2}\mid \lambda_{t_1}>\lambda_*\right]$ (left panel) and of $\mathbb{E}\!\left[q_{t_2}\mid q_{t_1}>q_*\right]$ on $t_2$.(right panel).} 
	\label{fig:two_figures_two_dates} 
\end{figure}

Finally, we evaluate the mixed two-date quantity
\[
\mathbb{E}\!\left[q_{t_2}\lambda_{t_2}\mid \lambda_{t_1}>\lambda_*\right],
\qquad 
\mathbb{E}\!\left[q_{t_2}\lambda_{t_2}\mid q_{t_1}>q_*\right].
\]

The numerical inversion gives
\[
\mathbb{E}\!\left[
q_{t_2}\lambda_{t_2}
\mathbf{1}_{\{\lambda_{t_1}>\lambda_*\}}
\right]\approx 0.378,
 \quad
\mathbb{E}\!\left[
q_{t_2}\lambda_{t_2}
\mathbf{1}_{\{q_{t_1}>q_*\}}
\right]\approx 0.306.
\]

Therefore,
\[
\mathbb{E}\!\left[
q_{t_2}\lambda_{t_2}
\,\middle|\,
\lambda_{t_1}>\lambda_*
\right]
=
\frac{
	\mathbb{E}\!\left[
	q_{t_2}\lambda_{t_2}
	\mathbf{1}_{\{\lambda_{t_1}>\lambda_*\}}
	\right]
}{
	\mathbb{P}(\lambda_{t_1}>\lambda_*)
}
\approx 6.187,
\qquad
\mathbb{E}\!\left[
q_{t_2}\lambda_{t_2}
\,\middle|\,
q_{t_1}>q_*
\right]
=
\frac{
	\mathbb{E}\!\left[
	q_{t_2}\lambda_{t_2}
	\mathbf{1}_{\{q_{t_1}>q_*\}}
	\right]
}{
	\mathbb{P}(q_{t_1}>q_*)
}
\approx 7.463.
\]
These quantity can be interpreted as a forward-looking stress indicator: it measures the combined level of congestion and arrival pressure at time $t_2$, conditional on the intensity process or the number of customers in the system being in its upper tail at the earlier time $t_1$.

Figure~\ref{fig:lambdatqt_2dates} displays the dependence of the two-date expected value
\[
\mathbb{E}\!\left[
q_{t_2}\lambda_{t_2}
\,\middle|\,
\lambda_{t_1}>\lambda_*
\right]
\]
on $t_2$ over the range $t_2 \in [1,8]$. The curve decays towards its long-term mean value given in \eqref{Eqn:lamdba_q_long_term} as
\begin{align}
	\frac{2\kappa^{2}\theta^{2}+2\kappa\theta^{2}\mu+\sigma^{2}\theta\mu}
{2\kappa\mu(\kappa+\mu)}
\approx 4.076.
\label{mix_long_team_mean}
\end{align}

Figure~\ref{fig:lambdatqt_2datesB} displays the dependence of the two-date expected value
\[
\mathbb{E}\!\left[
q_{t_2}\lambda_{t_2}
\,\middle|\,
q_{t_1}>q_*
\right]
\]
on $t_2$ over the range $t_2 \in [1,8]$. The curve also decays towards the long-term mean value found in \eqref{mix_long_team_mean}.

\begin{figure}[h!] 
	\centering 
	\begin{subfigure}[b]{0.48\textwidth} 
		\centering \includegraphics[width=\textwidth,height=5cm]{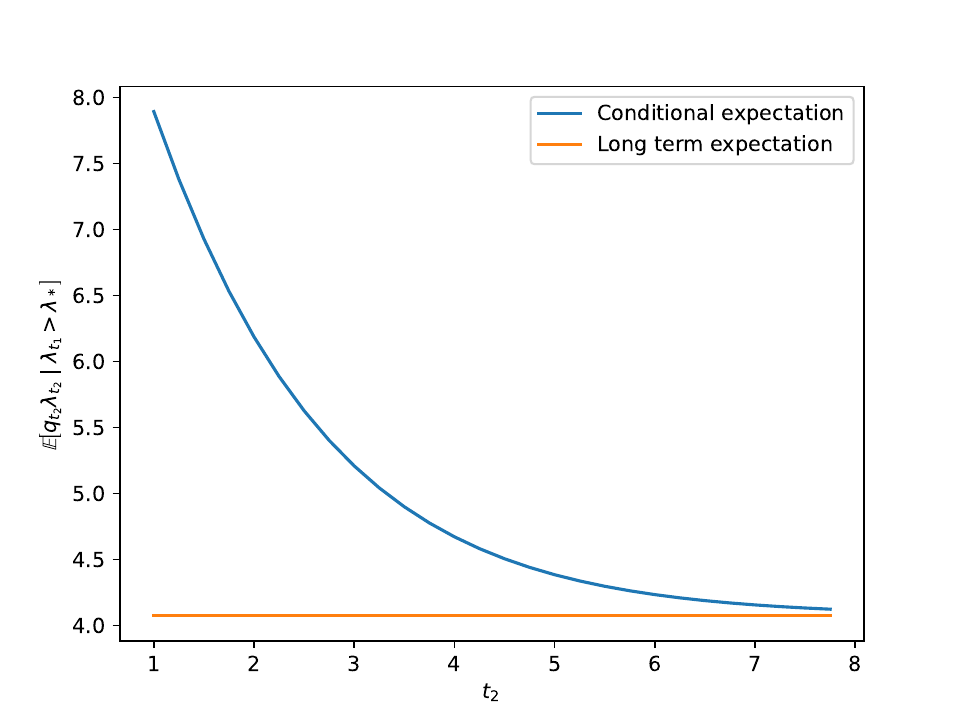} 
		\caption{Panel A} 
		\label{fig:lambdatqt_2dates} 
	\end{subfigure} 
	\hfill 
	\begin{subfigure}[b]{0.48\textwidth} 
		\centering \includegraphics[width=\textwidth,height=5cm]{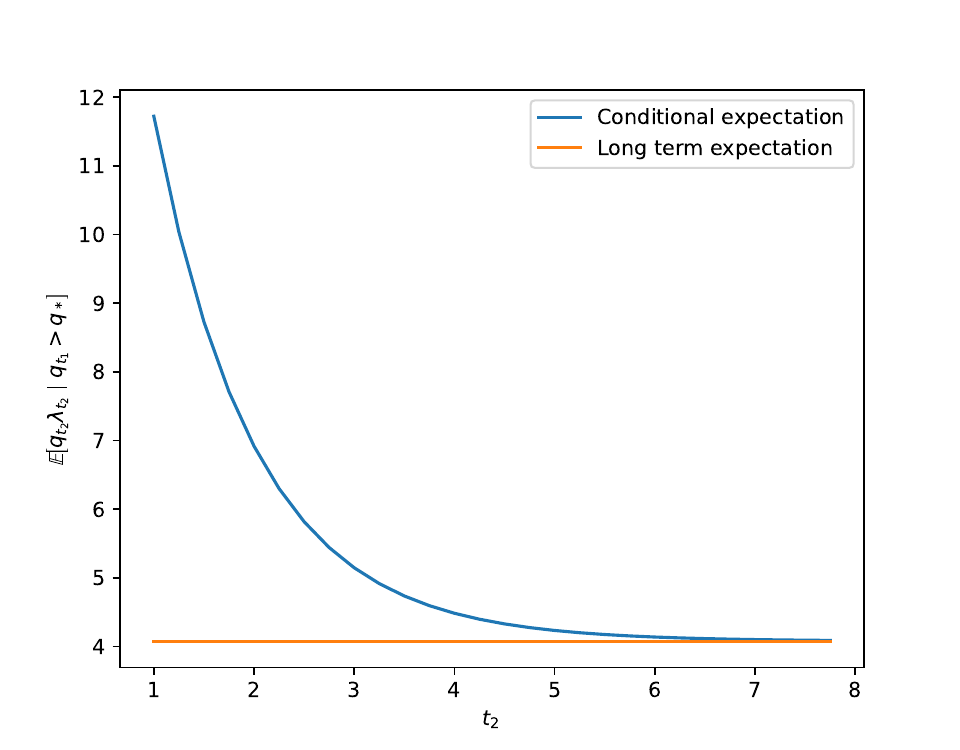} 
		\caption{Panel B} 
		\label{fig:lambdatqt_2datesB} 
	\end{subfigure} 
	\caption{Dependence of $\mathbb{E}\!\left[q_{t_2} \lambda_{t_2}\mid \lambda_{t_1}>\lambda_*\right]$  (left panel) and of $\mathbb{E}\!\left[q_{t_2} \lambda_{t_2}\mid q_{t_1}>q_*\right]$ on $t_2$ (right panel).} 
	\label{fig:two_figures_two_dates_product} 
\end{figure}

The above conditional expectations are computed using the one-dimensional risk measure representations given in Proposition~\ref{prop:JointExpectedShortfallDerivativeFourierTransform}, which significantly reduce the computational time compared with the two-dimensional risk measure representations given in Proposition~\ref{prop:JointExpectedShortfallFourierTransform}. Table~\ref{tab:counts} reports the resulting gain in computational speed.
\begin{table}[h]
	\centering
	\caption{Function call counts and computation time comparison.}
	\label{tab:counts}
	\begin{tabular}{c|c|c}
		\hline
		& counts 
		& time \\
		\hline
		1D integration for cross-risk measure $	\mathbb{E}[q_t \lambda_t \mathbf{1}_{\lbrace \lambda_t>\lambda_* \rbrace}]$ 		& $3363$ 
		& $4.03$ seconds \\ 
		2D integration for cross-risk measure $	\mathbb{E}[q_t \lambda_t \mathbf{1}_{\lbrace \lambda_t>\lambda_* \rbrace}]$ 
	& $6702$ 
	& $24.37$ seconds  \\
		\hline		
	\end{tabular}
\end{table}

\section{Conclusion}\label{sec:Conclusion}

In this paper, we develop risk measures for the infinite-server queue model driven by a Cox--BESQ process proposed in \citet{DaFonsecaShiZhang2026b}. The Cox--BESQ framework provides a tractable approach to modelling overdispersed arrival streams, for which the variability of the input process exceeds that of a classical Poisson process. A key feature of the model is its affine and explicitly solvable structure, which yields closed-form expressions for the MGF of the number of customers in the system at a given time and of its integrated counterpart.

Exploiting the same affine and solvable properties, we further show that the MGFs of the state variables at two distinct dates can also be obtained in closed form. These explicit MGFs allow us to derive Fourier representations of risk measures, both at a single date and across two dates. The resulting formulae require only one-dimensional numerical integration over the state variables and provide a more refined understanding of the behaviour of infinite-server queues under critical conditions. Finally, numerical experiments demonstrate how the proposed results can be implemented in practice.

% \clearpage
\footnotesize
\selectfont
\bibliographystyle{abbrvnat}
\bibliography{Biblio}

\normalsize
\selectfont
\clearpage
\appendix 
\setcounter{theorem}{0} 
\renewcommand{\thetheorem}{A.\arabic{theorem}}
%\section*{Appendix}

\begin{center}
 Online supplementary materials to:

\Large Risk measures for infinite-server queues driven by a Cox-BESQ process

\normalsize by
\large Jos\'e Da Fonseca and Yajun Shi and Wenjun Zhang

\end{center}
\normalsize

\setcounter{theorem}{0}

In order to prove Theorem \ref{theo:mgf_triplet}, Corollaries \ref{coro:mgf_lambda}, \ref{coro:mgf_q}, \ref{coro:q_infty}, \ref{coro:mgf_integral_q}, which appear in \citet*{DaFonsecaShiZhang2026b}, the following two well know results are needed.

\begin{proposition}\label{prop:Riccati_explosion}
Assume \((y(t))_{t\in \mathbb{R}_+}\) a function that satisfies the Riccati differential equation:
\begin{align}
y'(t)&=\frac{\sigma^2}{2}y(t)^2 -\kappa y(t) + c\,, \label{eq:Riccati_ode}\\
y(0) &=u,\, \label{eq:Riccati_initial_condition}
\end{align}
with \(\sigma>0\), \(\kappa>0\), \(u\in \mathbb{R}\) and \(c\in \mathbb{R}\) such that \(c\leq \kappa^2/(2\sigma^2)\). Then
\begin{align}
y(t) = - \frac{(c+u \nu_+) - (c+u \nu_-)e^{-\sqrt{\Delta}t} }{(\nu_{-}+u \sigma^2/2) - (\nu_{+}+u \sigma^2/2)e^{-\sqrt{\Delta}t}},  \label{eq:Riccati_solution}
\end{align}
with \(\Delta = \kappa^2 -2 \sigma^2 c\) and \(\nu_\pm =(-\kappa \pm \sqrt{\Delta})/2\) (with \(\nu_-<\nu_+\)). The explosion time \(t^*=\inf\lbrace t>0; y(t)=\infty\rbrace\), when it exists, is
\begin{equation}
t^* = \frac{1}{\sqrt{\Delta}} \ln\left(\frac{2\nu_+ + \sigma^2 u}{2\nu_- + \sigma^2 u} \right)\,. \label{eq:explosion_time_Riccati}
\end{equation}

There is no explosion whenever \( u \in ]-\infty,\, \frac{\kappa +\sqrt{\Delta}}{\sigma^2} [\).

\end{proposition}
\begin{proof}
Write \(y(t) = -h'(t)/(\sigma^2h(t)/2)\), then \(h(t)\) solves the ODE \(h''(t) + \kappa h'(t) + (c\sigma^2/2)h(t)=0\). Looking for a solution of the form \(h(t)=e^{\nu t}\), we reach that \(h(t)=c_- e^{\nu_- t} + c_+ e^{\nu_+ t}\) where \(\nu_\pm \) and \(\Delta\) as specified above while \(c_-,\,c_+\) are two constants. Using \eqref{eq:Riccati_initial_condition} we obtain \(c_{-}/c_{+}=-(2\nu_+ + u\sigma^2)/(2\nu_- + u\sigma^2)\), and from there we reach \eqref{eq:Riccati_solution}. The explosion time is obtain by making the denominator of \eqref{eq:Riccati_solution} null, it leads to \eqref{eq:explosion_time_Riccati}. If \(2\nu_+ + \sigma^2 u>0\) and \(2\nu_- + \sigma^2 u>0\), or equivalently \(u>(\kappa-\sqrt{\Delta})/\sigma^2\)  and  \(u>(\kappa+\sqrt{\Delta})/\sigma^2\) but the former is satisfied if the latter is, there is explosion and the explosion time is given by \eqref{eq:explosion_time_Riccati}. If \(2\nu_+ + \sigma^2 u<0\) and \(2\nu_- + \sigma^2 u<0\) there should be an explosion but since \(\nu_+>\nu_-\), or equvalently \(2\nu_+ + \sigma^2 u>2\nu_- + \sigma^2 u\), then \(2\nu_+ + \sigma^2 u<0\) and \(2\nu_- + \sigma^2 u<0\) would mean \((2\nu_+ + \sigma^2 u)/(2\nu_- + \sigma^2 u)<1\) and therefore a negative explosion time, so we conclude that there is no explosion in this case. Lastly, if \(2\nu_+ + \sigma^2 u>0\) and \(2\nu_- + \sigma^2 u<0\), or equvalently \((\kappa -\sqrt{\Delta})/\sigma^2 < u < (\kappa +\sqrt{\Delta})/\sigma^2\), then there is no explosion.     
\end{proof}

%%%%%%%%%%%
\begin{corollary}\label{coro:Riccati_explosion_b}
Consider \((y(t))_{t \in \mathbb{R}_+}\) satisfying the same equation as in Proposition \ref{prop:Riccati_explosion}. Assume that \((b(t))_{t \in \mathbb{R}_+}\) is a function that satisfies the Riccati differential equation
\begin{align}
	b'(t)&=\frac{\sigma^2}{2}b(t)^2 -\kappa b(t) + c(t)\,, \label{eq:Riccati_ode_b}\\
	b(0) &=u,\, \label{eq:Riccati_initial_condition_b}
\end{align}
with \(\sigma>0\), \(\kappa>0\), \(u\in \mathbb{R}\) and \(c\in \mathbb{R}\), \(c \geq c(t)\). Then
\[
y(t)\geq b(t).
\]
\begin{proof}
	Let \(z(t)=y(t)-b(t)\).  Its derivative satisfies  \(z'(t)=y'(t)-b'(t)=A(t)z(t)+(c-c(t))\) with \(A(t)=\frac{\sigma^2}{2}(y(t)+b(t))-\kappa\) whose solution is \(z(t)=e^{\int_0^t A(s)ds}\big(z(0)+ \int_0^t (c-c(\tau)) e^{-\int_0^\tau A(s)ds}d\tau \big)\). Since \(z(0)=y(0)-b(0)=0\) and \(c \geq c(t)\), the integral is non‑negative, giving  \(z(t)=y(t)-b(t) \geq 0 \). Hence, the conclusion follows.
\end{proof}
\end{corollary}

\begin{proof}[Proof of Theorem \ref{theo:mgf_triplet}]
	Using Feynman-Kac's formula, we know that the function \( (t,\lambda,q) \in ([0,\, T] \times \mathbb{R}_+  \times \mathbb{N}) \to f(t,T,\lambda,q)\) of \eqref{eq:mgf_triplet_expectation} satisfies the PDE: 
	\begin{equation*}
	\partial_t f(t,T,x) + \mathcal{L}f(t,T,x) + u_3 q f(t,T,x) = 0,
	\end{equation*}
	with \(x=(\lambda,q)\) and terminal condition \(f(T,T,x)=e^{u_1 \lambda + u_2 q}\). We look for an (exponential) affine solution of the form \eqref{eq:mgf_triplet}. Inserting this form in the above PDE, we obtain that \(t \to a(t)\), \(t \to b(t)\), and \(t \to d(t)\) satisfy the following set of ODEs:
	\begin{align}
		a'(t) &= \kappa \theta b(t), \label{eq:ode_a_triplet}\\
		b'(t) &= \frac{\sigma^2}{2}b(t)^2 -\kappa b(t) + (e^{d(t)}-1), \label{eq:ode_b_triplet}\\
		d'(t) &= \mu (e^{-d(t)}-1)+u_{3}, \label{eq:ode_d_triplet}
	\end{align}
	with the initial conditions \( a(0)=0, b(0)=u_1, d(0)= u_{2} \). Solving \eqref{eq:ode_d_triplet} yields \eqref{eq:sol_d_triplet}, while for \eqref{eq:ode_b_triplet}, we look for a solution of the form \(b(t)=-\frac{h'(t)}{\frac{\sigma^{2}}{2}h(t)}\), which means that the Riccati equation \eqref{eq:ode_b_triplet} is solvable, see \citet{GrasselliTebaldi2008}, and obtain that the function \(t \to h(t)\) solves an ODE of the form \(h''(t)+\kappa h'(t) +(\beta e^{(u_{3}-\mu) t}-\alpha )h(t)=0\). We then look for a solution of the form \(h(t)=e^{c_0t}g(c_1e^{c_2 t})\), after computing the derivatives and using the ODE satisfied by \(h\), we reach
	\footnotesize
	\begin{align}
		c_{1}c_{2}^{2}e^{c_{0}t}e^{c_{2}t}\left((c_{1}e^{c_{2}t}g'')+g'\left(\frac{c_{0}}{c_{2}}+\frac{c_{0}+c_{2}}{c_{2}}+\frac{\kappa}{c_{2}}\right)+g\left(\frac{c_{0}^{2}e^{-c_{2}t}}{c_{1}c_{2}^{2}}+\frac{\kappa c_{0}e^{-c_{2}t}}{c_{1}c_{2}^{2}}-\frac{\alpha e^{-c_{2}t}}{c_{1}c_{2}^{2}}+\frac{\beta e^{(u_{3}-\mu)t}e^{-c_{2}t}}{c_{1}c_{2}^{2}}\right)\right)=0, 
		\label{eq:g_ode_triplet}
	\end{align}
	\normalsize 
	with the argument of \(g\) being \(c_1e^{c_2 t}\). We choose \(c_0\) such that \(c_0^2+\kappa c_0-\alpha=0\), it leads to two set of parameters:
	\begin{align*}
		c_0 &= \frac{-\kappa-\sqrt{\kappa^2 + 4\alpha}}{2}, \; c_2 = u_{3} -\mu, \; c_1=-\frac{\beta}{c_2^2},\\
		c_0 &= \frac{-\kappa+\sqrt{\kappa^2 + 4\alpha}}{2}, \; c_2 = u_{3}-\mu, \; c_1=-\frac{\beta}{c_2^2},
	\end{align*}
	both imply a coefficient for \(g\) in the ODE \eqref{eq:g_ode_triplet} of \(-1\) and in the first case the coefficient of \(g'\) is \(1+k_1\) where \(k_1\) is given by \eqref{eq:k1_alpha_beta_triplet} so that \(g\) satisfies the ODE \(zg''(z) + (1+k_1) g'(z) - g(z)=0\) whose solution is \eqref{eq:sol_h_1_triplet}. In the second case, the coefficient of \(g'\) is \(1-k_1\) so that \(g\) satisfies the ODE \(zg''(z) + (1-k_1) g'(z) - g(z)=0\) whose solution is \eqref{eq:sol_h_2_triplet}. Using the property of the function, we deduce that the derivatives of these functions are given by \eqref{eq:sol_h_1_triplet_prime}  and \eqref{eq:sol_h_2_triplet_prime}. As a result, the function \(b(t)\) is 
	\begin{align*}
		b(t) &=-\frac{2}{\sigma^2}\frac{ \bar{c}_1 h_1'(t)+ \bar{c}_2h_2'(t)}{\bar{c}_1 h_1(t)+ \bar{c}_2h_2(t)}, 
	\end{align*}

	for two constants \(\bar{c}_1,  \bar{c}_2\) and using the boundary condition \(b(0)=u_1\), we obtain that \(\frac{\bar{c}_1}{\bar{c}_2}=-\frac{\frac{2}{\sigma^2} h_2'(0) + u_1 h_2(0)}{\frac{2}{\sigma^2} h_1'(0) + u_1 h_1(0)}\). Inserting this condition in the above equation we deduce \eqref{eq:sol_b_triplet}. Integrating \eqref{eq:ode_a_triplet}, taking into account that \(b(t)\) is a logarithmic derivative, gives \eqref{eq:sol_a_triplet}. Note that \(\kappa^2+4\alpha>0\) with \(\alpha = \frac{\sigma^2}{2}\frac{u_{3}}{u_{3}-\mu}\) is satisfied if \(u_3 <\mu/(1+2\sigma^2/\kappa^2)\).
	
We look for \(t^*>0\) and positive constants \(u_1^*,\, u_2^*,\, u_3^*\), such that for \(u_1<u_1^*,\, u_2<u_2^*,\, u_3<u_3^*\) and \(t<t^*\) the MGF is well defined. Indeed, consider \(d(t)=\ln(g(u_2,u_3,t))\) with \(g(u_2,u_3,t)=e^{(u_{3}-\mu)t}e^{u_{2}}  -  \frac{\mu}{u_{3}-\mu}\left(1-e^{(u_{3}-\mu)t}\right)\). Define \(u_3^*=\mu\), then for \(u_3<u_3^*\) and \(t\geq 0 \), \(g(u_2,u_3,t)\) is positive so that \(d(t)\) of \eqref{eq:sol_d_triplet} is well defined. The function \(u_2 \to g(u_2,u_3,t)-1\) is increasing and there exist \(u_2^*>0\) and \(t^*>0\) such that for \(u_2<u_2^*\) and \(t<t^*\) (and \(u_3<u_3^*\)) we have \(g(u_2,u_3,t)-1<\kappa^2/(2\sigma^2)\). Denote \(c(t)= g(u_2,u_3,t)-1\), then \eqref{eq:ode_b_triplet} rewrites as \(b'(t)=\frac{\sigma^2}{2}b(t)^2 -\kappa b(t) + c(t)\). Define \(y_t\) solution of  \eqref{eq:Riccati_ode} with \(y(0) =u_1\) with \(c = \kappa^2/(2\sigma^2)\) and \(u_1<u_1^*:= \frac{\kappa +\sqrt{\Delta}}{\sigma^2}\) then for \(t<t^*\) the function \(y_t\) does not explode and from Proposition \ref{coro:Riccati_explosion_b} we conclude that for \(t<t^*\) the function \(b(t)\) also does not explode. Lastly, \(a(t)\) solving \eqref{eq:ode_a_triplet} is obtained by integration, as result for \(t<t^*\) the function \(a(t)\)  is also well defined. It leads to the announced result.	
\end{proof}

%%%%%%%%%%%%%%%%%%%%%%%%%%%%%%%%%%%%%%%%%%%%%%%%%%%%%%%%%%%%%%

\begin{proof}[Proof of Corollary \ref{coro:mgf_q}]
	Setting \(u_1=0\), \(u_2=u\) and \(u_3=0\) in Theorem \ref{theo:mgf_triplet} we obtain that \(k_1=\kappa/\mu\), \(\alpha=0\) and \(\beta=\frac{\sigma^2}{2}(e^{u}-1)\), and we obtain the above equations after simplifications.
	Consider \(d(t)=\ln(g(u,t))\) with \(g(u,t)=1+e^{-\mu t}(e^{u}-1)\). For \(t\geq 0 \) and \(u\in \mathbb{R}\), the function \(g(u,t)\) is positive so that \(d(t)\) in \eqref{eq:sol_d_q} is well defined. Since \(u \to g(u,t)-1\) is increasing, there exists \(u^*>0\) and \(t^*>0\) such that \(g(u,t)-1<\kappa^2/(2\sigma^2)\) for \(u<u^*\) and \(t<t^*\). Denote \(c(t)= g(u,t)-1\), then \eqref{eq:ode_b_triplet} becomes  \(b'(t)=\frac{\sigma^2}{2}b(t)^2 -\kappa b(t) + c(t)\). Let \(y_t\) be the solution of  \eqref{eq:Riccati_ode} with \(y(0) =0\) and \(c = \kappa^2/(2\sigma^2)\). For \(t<t^*\), the function \(y_t\) does not explode, and by Proposition \ref{coro:Riccati_explosion_b}, the same holds for \(b(t)\). Finally, \(a(t)\), obtained by integrating \eqref{eq:ode_a_triplet}, is also well defined. Hence, the claimed result follows.
\end{proof}

%%%%%%%%%%%%%%%%%%%%%%%%%%%%%%%%%%%%%%%%%%%%%%%%%%%%%%%%%%%%%%
\begin{proof}[Proof of Corollary \ref{coro:q_infty}]

	Since \(\mu>0\) then \(\lim_{t \to +\infty}d(t)=0\) in \eqref{eq:sol_d_q}. Using the expression \eqref{eq:sol_h_1_q}, since \(\kappa>0\), \(\mu>0\) and \(\lim_{z\to 0}{}_0F_1(.,z)=1\) we obtain that \(\lim_{t\to +\infty}h_1(t)=0\), \(\lim_{t\to +\infty}h_1'(t)=0\), \(\lim_{t\to +\infty}h_2'(t)=0\) and \(\lim_{t\to +\infty}h_2(t)=1\). As a result,  \(\lim_{t\to +\infty }b(t)=0\). Still using these limits, from \eqref{eq:sol_a_q} we reach that \(\lim_{t\to +\infty}a(t)\) is well defined and equal to \eqref{eq:sol_a_infty}. 
	
	To clarify that there is not explosion, since \(a(t)\) is obtained by integrating the function \(b(t)\) over an infinite interval (see \eqref{eq:ode_a_triplet}), let us consider \(d(t)=\ln(g(u,t))\) with \(g(u,t)=1+e^{-\mu t}(e^{u}-1)\). For \(t\to \infty \) and \(u\in \mathbb{R}\), the function \(g(u,t)\) tends to \(1\) so that \(d(t)\) in \eqref{eq:sol_d_q} is well defined when \(t\to +\infty\), and \(\lim_{t\to \infty}e^{d(t)}-1 \to 0\). Hence, \eqref{eq:ode_b_triplet} becomes  \(b'(t)=\frac{\sigma^2}{2}b(t)^2 -\kappa b(t)\) for \(t\) large, and the function \(b(t)\) does not explode by Corollary \ref{coro:Riccati_explosion_b}. What is more, taking in Proposition \ref{prop:Riccati_explosion} the constant \(c=0\) in \eqref{eq:Riccati_solution}, we reach that there exists \(u>0\) and \(c_1>0\) such that \(y(t) \sim c_1 e^{-\kappa t} \quad t \to +\infty\) and therefore it is integrable (with respect to \(t\)) over \(\mathbb{R}_+\), so is \(b(t)\), and therefore \(\lim_{t\to +\infty}a(t)\) is well defined and is equal to \eqref{eq:sol_a_infty}.
		
	The last statement is straightforward.
\end{proof}

%%%%%%%%%%%%%%%%%%%%%%%%%%%%%%%%%%%%%%%%%%%%%%%%%%%%%%%%%%%%%%
\begin{proof}[Proof of Corollary \ref{coro:mgf_integral_q}]\label{proof:q_infty}
	Let \(g(u,t)=1- \frac{u}{u-\mu}(1-e^{(u-\mu) t})\). Set \(u^*=\mu\). For \(u<u^*\) and \(t\geq 0 \), the function \(g(u,t)\) is positive thus \(d(t)=\ln(g(u,t))\) of \eqref{eq:sol_d_integral_q} is well defined. Moreover, \(u \to g(u,t)-1\) is increasing and there exist \(u^*>0\) and \(t^*>0\) such that \(c(t):= g(u,t)-1<\kappa^2/(2\sigma^2)\)  for \(u<u^*\) and \(t<t^*\). With this \(c(t)\), equation \eqref{eq:sol_b_integral_q} takes the Riccati form \(b'(t)=\frac{\sigma^2}{2}b(t)^2 -\kappa b(t) + c(t)\). The function \(c(t)\) satisfies the same bound as in Proposition \ref{coro:Riccati_explosion_b}, the argument of the proposition (using a comparison Riccati equation with a constant \(c=\kappa^2/(2\sigma^2)\)) shows that \(b(t)\) dose not explode. Consequently the integral defining \(a(t)\) also converges.	  
\end{proof}

\end{document}